\documentclass[final]{siamltex}
\usepackage{amsfonts}
\usepackage{amsmath,amssymb,epsfig,
color
}

\makeatletter
\@addtoreset{equation}{section}
\renewcommand{\theequation}{\thesection.\arabic{equation}}
\makeatother

\newtheorem{example}[theorem]{Example}
\newcommand{\tr}{\mathop{\mathrm{tr}}\limits}

\newcommand{\bff}{{ f}}

\newcommand{\bx}{{x}}

\newcommand{\bv}{{v}}

\newcommand{\range}{\operatorname{range}}

\newcommand{\cF}{{\mathcal F}}

\newcommand{\cH}{{\mathcal H}}

\newcommand{\cL}{{\mathcal L}}

\def\R{ \mathbb {R}}
\def\S{ \mathbb {S}}

\newcommand{\N}{\mathbb{N}}

\begin{document}

\sloppy

\newcounter{fig}

\title{Kernel-based Discretisation  for Solving  Matrix-Valued PDEs}

\author{Peter Giesl\thanks{Department of Mathematics, University of
    Sussex, Falmer BN1 9QH, United Kingdom, {\tt p.a.giesl@sussex.ac.uk}}
         \and
Holger Wendland\thanks{Applied and Numerical Analysis, Department of
  Mathematics, University of 
  Bayreuth, 95440 Bayreuth, Germany, {\tt holger.wendland@uni-bayreuth.de}}}

\date{\today}

\maketitle
\begin{abstract}
In this paper, we discuss the solution of certain matrix-valued
partial differential equations. Such PDEs arise, for example, when
constructing a Riemannian contraction metric for a dynamical system given by an autonomous ODE. We develop and analyse a new meshfree discretisation scheme using
kernel-based approximation spaces. However, since these approximation
spaces have now to be matrix-valued, the kernels we need to use are
fourth order tensors. We will review and extend recent results on even
more general reproducing kernel Hilbert spaces. We will then apply
this general theory to solve a matrix-valued PDE and derive
error estimates for the approximate solution. The paper ends with a
typical example from dynamical systems.
\vspace*{1ex}



\end{abstract}

\section{Introduction}

Kernel-based discretisation methods provide an extremely flexible,
general framework to approximate the solution to even rather unconventional
problems (see for example
\cite{Buhmann-03-1,
  Berlinet-Thomas-Agnan-04-1,Wendland-05-1,Fasshauer-07-1,Fornberg-Flyer-15-1, 
  Schaback-Wendland-06-1}). They are meshfree methods, requiring only a
discrete data set for discretising the underlying domain. Since the
kernel can be chosen problem dependent, it is very easy to construct
in particular smooth approximation spaces and high order methods. 

Kernel-based methods have extensively been used for solving partial
differential equation (see for example
\cite{Franke-Schaback-98-2,Kansa-90-1,Flyer-Fornberg-11-1,Wendland-99-2}). They have been used in the context of dynamical systems for constructing
Lyapunov functions (\cite{Giesl-07-1, Giesl-Wendland-07-1}) and they also
play a key role in learning theory
(\cite{Cristianini-Shawe-Taylor-00-1, Cucker-Smale-02-1,
  Micchelli-Pontil-05-1, Schoelkopf-Smola-02-1, Vapnik-98-1,
  Steinwart-Christmann-08-1}) and high-dimensional integration
(see for example \cite{Dick-etal-13-1}) and many other areas.

It is the goal of this paper to derive and analyse a new method for
reconstructing matrix-valued functions
$M:\Omega\subseteq\R^n\to\R^{n\times n}$ from a matrix-valued PDE of the
form 
\begin{equation}\label{general-form}
Df^T(x)M(x)+M(x)Df(x)+M'(x)=-C(x), \qquad x\in\Omega\subseteq\R^n.
\end{equation}
Here, $f$ and $C$ are given functions, defined on a given domain $\Omega\subseteq\R^n$. The function 
$f:\Omega\to\R^n$ is a differentiable vector-valued function with
derivative matrix $Df:\Omega \to \R^{n\times n}$, $C:\Omega \to\R^{n\times n}$ is
a matrix-valued function and $M'$ is 
the so-called orbital derivative, i.e. it is component-wise defined to
be $(M'(x))_{ij} = \nabla M(x)_{ij}\cdot f(x)$. 

Our work is motivated by a typical application of such matrix-valued
PDEs from the theory of dynamical systems. To be more precise, when
studying autonomous 
ODEs of the form $\dot x=f(x)$, then an exponentially stable
equilibrium can be characterised by a Riemannian contraction metric
\cite{Giesl-15-1}:

\begin{theorem}\label{th:1}
Let $\emptyset\not= G\subseteq \R^n$ be a compact, connected and
positively invariant set and $M$ be a Riemannian contraction metric in
$G$, i.e. 
\begin{itemize}
\item $M\in C^1(G,\R^{n\times n})$, such that $M(\bx)$ is
  symmetric and positive definite for all $\bx\in G$. 
\item $D\bff(\bx)^TM(\bx)+M(\bx)D\bff(\bx)+M'(\bx)$ is
  negative definite for all $\bx\in G$.
\end{itemize}
Then there exists one and only one equilibrium in $\bx_0$ in $G$;
$\bx_0$ is exponentially stable and $G$ is a subset of the basin of
attraction $A(\bx_0)$. 
\end{theorem}

The difficulty of this approach is to constructively find 
such a contraction metric. In 
\cite{Giesl-15-1} a contraction metric is characterised as the
solution of a first-order PDE of the form (\ref{general-form})
for all $\bx\in \Omega=A(x_0)$,  where $C(x)=C\in
\R^{n\times n}$ is a given symmetric and  positive definite matrix.  The
construction of $M$ is thus a typical example of a matrix-valued PDE
with the additional complication that the solution also has to be
symmetric and positive definite. 

In the accompanying paper \cite{Giesl-Wendland-??-1}, we will prove the
theoretical results required in the dynamical system context. In this
paper, however, we will concentrate on deriving the numerical framework for
discretising even more general PDEs of the form
\begin{equation}\label{fmgeneral}
F(M)(x)=-C(x),\qquad x\in\Omega,
\end{equation}
where $F$ is a rather general differential operator which maps
matrix-valued Sobolev functions of order $\sigma$ to matrix-valued
Sobolev functions of order $\tau<\sigma$. 

The paper is organised as follows.
In Section \ref{sec2} we will review and extend results on general
reproducing kernel Hilbert spaces, going far beyond the usual
definition. In Section \ref{sec:OptRec} we generalise the theory of
optimal recovery to these general reproducing kernel Hilbert
spaces. In Section \ref{sec:2.3} we will become more concrete by
restricting ourselves to reproducing kernel Hilbert spaces of
matrix-valued functions. In Section \ref{sec:2.4} we continue this by
looking at Sobolev spaces of matrix-valued functions. In Section
\ref{sec:2.5} we will derive error estimates for optimal recovery
processes of solutions to (\ref{fmgeneral}). Section \ref{sec3} then
deals with the application to the above mentioned problem to construct
a contraction metric for an autonomous system. The final section gives
a numerical example.

\section{Reproducing Kernel Hilbert Spaces}
\label{sec2}

Reproducing Kernel Hilbert Spaces have first been introduced to
describe real-valued functions $f\colon \Omega\to\R$ on a domain
$\Omega\subseteq \R^d$ (see for example \cite{Aronszajn-50-1}). They
require a kernel $\Phi:\Omega\times \Omega \to \R$ with the
reproduction property 
$f(x)=\langle f, \Phi(\cdot,x)\rangle_\cH$ for $f\in \cH$,
$x\in\Omega$ where $\cH$ denotes a
Hilbert space of functions $f:\Omega\to \R$. 
Later, so-called matrix-valued kernels $\Phi\colon \Omega\times \Omega
\to \R^{n\times n}$  with the reproduction property $f(x)^T\alpha
=\langle f, \Phi(\cdot,x)\alpha \rangle_\cH$, have been introduced to recover
vector-valued functions $f\colon \Omega\to\R^n$ where $\cH$ denotes a
Hilbert space of functions $\Omega\to \R^n$ and $\alpha\in\R^n$ is an arbitrary
vector (see for example \cite{Amodei-97-1,Benbourhim-Bouhamidi-10-1,
  Fuselier-08-1, Lowitzsch-02-1,Narcowich-Ward-94-2,Wendland-09-1}).

We are interested in reproducing kernel Hilbert spaces of
matrix-valued functions. We start with a more general introduction,
namely functions with values in a general Hilbert space $W$, which in
the above examples was $\R$ and $\R^n$, respectively,
and will later be the space $\R^{n\times n}$ of all real $n\times n$
matrices or the space $\S^{n\times n}$ of all symmetric real $n\times
n$ matrices. 
The notion of such general reproducing kernel Hilbert spaces is 
not new, see for example \cite{Micchelli-Pontil-05-1} and the
literature therein.

Let $W$ be a real Hilbert space and denote the linear space of all
linear and bounded operators $L:W\to W$ by $\cL(W)$. For any $L\in
\cL(W)$, we will denote the adjoint operator by $L^*\in \cL(W)$.

Let $\Omega\subseteq \R^d$ be a given domain and let $\cH(\Omega;W)$ be
a Hilbert space of $W$-valued functions $f:\Omega\to W$.

\begin{definition}
\label{def:RKHS} The Hilbert space $\cH(\Omega;W)$ is called a {\em
    reproducing kernel Hilbert space} if there is a function
  $\Phi:\Omega\times\Omega\to \cL(W)$  with
\begin{enumerate}
\item $\Phi(\cdot,x)\alpha \in \cH(\Omega;W)$ for all $x\in\Omega$ and
  all $\alpha\in W$.
\item $\langle f(x),\alpha\rangle_W = \langle f,
  \Phi(\cdot,x)\alpha\rangle_{\cH}$ for all $f\in\cH(\Omega;W)$, all
  $x\in\Omega$ and all $\alpha\in W$.
\end{enumerate}
The function $\Phi$ is called the {\em reproducing kernel} of
$\cH(\Omega;W)$.
\end{definition}

Let us have a short look at two typical examples that have been
introduced at the beginning of this section.

\begin{example}
If we choose $W=\R$ with inner product being just the product, then
$\cH(\Omega;W)$ consists of real-valued functions. Moreover, each
element $L$ of $\cL(\R)$
can be represented by $Lw=\ell w$ with $\ell\in\R$ and thus $\cL(\R)$
can be identified with $\R$. Hence, a reproducing kernel in this
setting has to satisfy $\Phi(\cdot,x)\alpha \in \cH(\Omega;W)$ for all
$\alpha\in\R$, which is obviously equivalent to $\Phi(\cdot,x)\in
\cH(\Omega;W)$, and the second condition is equivalent to $f(x) =
\langle f,\Phi(\cdot,x)\rangle_\cH$. Hence, this is the classical
reproducing kernel used in approximation theory and other areas.
\end{example}

\begin{example}
If we choose $W=\R^n$ with the standard inner product, then $\Phi$ has
to map into the linear mappings from $\R^n\to\R^n$ and can hence be
represented by a matrix. Thus, $\Phi(\cdot,x)\alpha$ represents now a
vector-valued function and the second condition in the definition
becomes
\[
f(x)^T\alpha = \langle f,\Phi(\cdot,x)\alpha\rangle_\cH.
\]
This is usually referred to as matrix-valued kernels in the
literature.
\end{example}

Before we come to our specific situation, we want to point out a few
general results, see also \cite{Micchelli-Pontil-05-1}.

\begin{lemma}\label{le:RKHS} 
\begin{enumerate}
\item The reproducing kernel $\Phi$ of a Hilbert space $\cH(\Omega;W)$
  is uniquely determined.
\item The reproducing kernel satisfies $\Phi(x,y)^*=\Phi(y,x)$ for all
  $x,y\in \Omega$.
\item The reproducing kernel is positive semi-definite,
  i.e. it satisfies
\[
 \sum_{i,j=1}^N\left\langle
\alpha_i,\Phi(x_i,x_j)\alpha_j\right\rangle_W \ge 0
\]
for all $x_1,\ldots,x_N\in\Omega$ and all $\alpha_1,\ldots,\alpha_N\in W$.
\end{enumerate}
\end{lemma}
\begin{proof}
The first property is proven as in the classical reproducing kernel
setting by assuming that there are two kernels and showing that they
have to be the same using the reproduction property.
%
The second property follows by setting $f=\Phi(\cdot,y)\beta$ in the
reproduction formula. This yields
\begin{eqnarray*}
\langle \Phi(x,y)\beta,\alpha\rangle_W &=& \langle
\Phi(\cdot,y)\beta,\Phi(\cdot,x)\alpha\rangle_\cH=\langle
\Phi(\cdot,x)\alpha,\Phi(\cdot,y)\beta\rangle_\cH\\
& =& \langle
\Phi(y,x)\alpha,\beta\rangle_W = \langle \beta, \Phi(y,x)\alpha\rangle_W.
\end{eqnarray*}
The third property simply follows from
\begin{eqnarray*}
\sum_{i,j=1}^N\left\langle
\alpha_i,\Phi(x_i,x_j)\alpha_j\right\rangle_W
&=&
 \sum_{i,j=1}^N\left\langle
\Phi(x_j,x_i)\alpha_i,\alpha_j\right\rangle_W 
 = \sum_{i,j=1}^N
\langle \Phi(\cdot,x_i)\alpha_i, \Phi(\cdot,x_j)\alpha_j\rangle_\cH \\
&=&
\left\|\sum_{i=1}^N \Phi(\cdot,x_i)\alpha_i\right\|_{\cH}^2 \ge 0.
\end{eqnarray*}
\hfill\end{proof}

In most cases, the kernel is even positive definite, namely if the
functions $\Phi(\cdot,x_j)\alpha_j$ are linearly independent.

\begin{definition}\label{def:posdef}
A kernel $\Phi:\Omega\times\Omega\to\cL(W)$ which satisfies
$\Phi(x,y)^*=\Phi(y,x)$ for all $x,y\in\Omega$ is called {\em positive
definite} if for all $N\in\N$, for all $x_1,\ldots,x_N\in\Omega$, pairwise
distinct, and for all $\alpha_1,\ldots,\alpha_N\in W$, not all of them
zero, we have
\[
 \sum_{i,j=1}^N\left\langle
\alpha_i,\Phi(x_i,x_j)\alpha_j\right\rangle_W >0.
\]
\end{definition}

As usual in the theory of reproducing kernel Hilbert spaces, it is
also possible to start with a kernel and to build its Hilbert space
from scratch. This is done as follows. Suppose we have a positive
definite kernel $\Phi:\Omega\times\Omega\to \cL(W)$ as in
Definition \ref{def:posdef}. Then, we can form the space
\[
\cF_\Phi(\Omega;W)=\mbox{span}\left\{
\Phi(\cdot,x)\alpha : x\in\Omega, \alpha\in W\right\}
\]
and equip this space with an inner product defined by
\[
\langle \Phi(\cdot,x)\alpha, \Phi(\cdot,y)\beta\rangle_\Phi:=\langle
\Phi(x,y)\beta, \alpha\rangle_W.
\]
The closure of $\cF_\Phi(\Omega;W)$ with respect to the norm induced by
this inner product is then the corresponding Hilbert space
$\cH(\Omega;W)$ for which $\Phi$ is the reproducing kernel.

\section{Optimal Recovery}\label{sec:OptRec}

If $\Phi:\Omega\times\Omega\to \cL(W)$ is a positive definite kernel,
then this immediately implies that we can solve the following
interpolation problem uniquely.

\begin{theorem}
If $x_1,\ldots,x_N$ are pairwise distinct points from $\Omega$ and if
$f_1,\ldots, f_N\in W$ are given, then there is exactly one
interpolant of the form
\[
s_f(x) = \sum_{j=1}^N \Phi(x,x_j)\alpha_j
\]
which satisfies $s_f(x_i)=f_i$, $1\le i\le N$.
\end{theorem}

\begin{proof}
Let $W^N$ denote the Cartesian product of the Hilbert space $W$. Then,
$W^N$ becomes a Hilbert space itself if equipped with the inner
product
\[
\langle \alpha,\beta\rangle_{W^N}= \sum_{j=1}^N \langle
\alpha_j,\beta_j\rangle_W.
\]
The {\em matrix} $A:=(\Phi(x_i,x_j))_{1\le i,j\le N}$ defines a linear
mapping $A:W^N\to W^N$, which is self-adjoint because of the second statement in
Lemma \ref{le:RKHS}: 
\begin{eqnarray*}
\langle A \alpha,\beta\rangle_{W^N} & = & \sum_{i=1}^N \langle
(A\alpha)_i,\beta_i\rangle_W = \sum_{i=1}^N\sum_{j=1}^N \langle 
\Phi(x_i,x_j)\alpha_j,\beta_i\rangle_W\\
& = & \sum_{i=1}^N \sum_{j=1}^N \langle \alpha_j,
\Phi(x_j,x_i)\beta_i\rangle_W
= \sum_{j=1}^N \langle \alpha_j, (A\beta)_j\rangle_W\\
& = & \langle \alpha,A\beta\rangle_{W^N}.
\end{eqnarray*}
Thus, the relation $\ker(A^*)=\range(A)^\bot$ shows together with
$A=A^*$ that $A$ is injective if and only if $A$ is surjective. 
But injectivity follows directly from the fact that $\Phi$ is positive
definite. 
\end{proof}

Within this general framework, we now want to discuss the more general
concept of optimal recovery. Hence, let $\cH(\Omega;W)$ be our
reproducing kernel
Hilbert space with reproducing kernel $\Phi:\Omega\times\Omega\to
\cL(W)$. As usual, we denote the dual of $\cH(\Omega;W)$ by
$\cH(\Omega;W)^*$.

\begin{definition}\label{optrec}
Given $N$ linearly independent functionals
$\lambda_1,\ldots,\lambda_N\in \cH(\Omega;W)^*$ and $N$ values
$f_1=\lambda_1(f),\ldots, f_N=\lambda_N(f)\in\R$ generated by an
element $f\in\cH(\Omega;W)$. The optimal recovery of $f$ based on this
information is defined to be the element $s^*\in\cH(\Omega;W)$ which
solves
\[
\min \left\{\|s\|_\cH : s\in\cH(\Omega;W) \mbox{ with
}\lambda_j(s)=f_j, 1\le j\le N\right\}.
\]
\end{definition}
The solution to this minimisation problem is well-known and follows
directly from standard Hilbert space theory; it works in any Hilbert
space, not 
only in reproducing kernel Hilbert spaces. We quote the following
result from \cite[Theorem 16.1]{Wendland-05-1}:
\begin{theorem}\label{th:mini}
Let $H$ be a Hilbert space. Let
  $\lambda_1,\ldots,\lambda_N\in H^*$ be linearly independent linear
  functionals with Riesz representers $v_1,\ldots,v_N\in H$. Then
  the element $s^*\in H $ which solves
\[
\min \{\|s\|_H : s\in H \mbox{ with } \lambda_j(s) = f_j, 1\le j\le
N\}
\]
is given by
\[
s^* = \sum_{k=1}^N \beta_k v_k,
\]
where the coefficients $\beta_k\in\R$ are determined by the
generalised interpolation conditions $\lambda_i(s^*)=f_i$, $1\le i\le
N$, which lead to the linear system $A_\Lambda \beta = f$ with the
positive definite matrix $A_\Lambda = (a_{ik})$ having entries
$a_{ik} = \lambda_i(v_k)=\langle v_k,v_i\rangle_H$.
\end{theorem}

Returning to our specific situation $H=\cH(\Omega;W)$, to apply this
theorem, we will need to know the Riesz representers of our
functionals $\lambda\in \cH(\Omega;W)^*$.  We start with rather
specific functionals.

\begin{lemma} Let $\lambda\in \cH(\Omega;W)^*$ be of the form
$\lambda(f)=\langle f(x),\alpha\rangle_W$, $f\in \cH(\Omega;W)$
with fixed $x\in\Omega$ and $\alpha\in W$. Then,
$\lambda=\lambda_{x,\alpha}$ has the Riesz representer
\[
v_\lambda = \Phi(\cdot,x)\alpha\in \cH(\Omega;W).
\]
\end{lemma}
\begin{proof}
This simply follows from applying $\lambda_{x,\alpha}$ to the specific
function $f=\Phi(\cdot,y)\beta$ with $y\in\Omega$ and $\beta\in
W$. Using the definition of the functional and the reproducing kernel
property yields
\[
\lambda_{x,\alpha}(\Phi(\cdot,y)\beta)  =  \langle
\Phi(x,y)\beta,\alpha\rangle_W 
 =  \langle \Phi(\cdot,y)\beta,\Phi(\cdot,x)\alpha\rangle_{\cH}.
\]
However, we also have by the Riesz representation theorem that
\[
\lambda_{x,\alpha}(\Phi(\cdot,y)\beta) = \langle \Phi(\cdot,y)\beta,
v_\lambda\rangle_{\cH}.
\]
Since the functions $\Phi(\cdot,y)\beta$ are dense in $\cH(\Omega;W)$,
this gives $v_\lambda = \Phi(\cdot,x)\alpha$.
\end{proof}

The result for arbitrary functionals can be reduced to this special case.

\begin{proposition}\label{pro:Riesz}
Assume that $\{\alpha_j\}_{j\in J}$ is an orthonormal basis of
$W$. Then, the Riesz representer of a functional $\lambda\in
\cH(\Omega;W)^*$ is given by
\[
v_\lambda(x) = \sum_{j\in J} \lambda(\Phi(\cdot,x)\alpha_j)\alpha_j,
\qquad x\in \Omega.
\]
\end{proposition}
\begin{proof}
Since $v_\lambda(x)\in W$ for every $x\in\Omega$ and since
$\{\alpha_j\}_{j\in J}$ is an orthonormal basis of $W$, we can expand
$v_\lambda(x)$ within this basis using its Fourier representation
\[
v_\lambda(x) = \sum_{j\in J} \langle v_\lambda(x),\alpha_j\rangle_W
\alpha_j.
\]
The result then follows immediately from the reproducing kernel
property:
\[
\langle v_\lambda(x),\alpha_j\rangle_W = \langle v_\lambda,
\Phi(\cdot,x)\alpha_j\rangle_{\cH} =\langle
\Phi(\cdot,x)\alpha_j,v_\lambda\rangle_{\cH} =\lambda(\Phi(\cdot,x)\alpha_j).
\]
\hfill\end{proof}

Thus, the optimal recovery problem can be recast as a linear
system. From now on, we will write $\lambda^y(\Phi(y,x)\alpha)$ to
indicate that the functional $\lambda$ acts on the variable $y$ of
the kernel.

\begin{corollary}\label{cor:Riesz}
Assume that $\{\alpha_j\}_{j\in J}$ is an orthonormal basis of $W$.
The solution of the minimisation problem of Theorem \ref{th:mini} is given by
\[
s^* = \sum_{k=1}^N \beta_k \sum_{j\in J} \lambda_k^y
(\Phi(y,\cdot)\alpha_j)\alpha_j,
\]
and the coefficients $\beta_k\in\R$ are determined by
\[
\sum_{k=1}^N\lambda_i^x\left[\lambda_k^y \sum_{j\in J}\left(
  \Phi(y,x)\alpha_j\right)\alpha_j\right] \beta_k = f_i, \qquad 1\le
i\le N.
\]
\end{corollary}

\begin{example}
Let us again have look at vector-valued functions, i.e. we let
$W=\R^n$. Then, we can choose the standard basis
$(e_j)_{j=1,\ldots,n}$ of $\R^n$ as the 
orthonormal basis and hence, the optimal recovery is given by
\[
s^* = \sum_{k=1}^N \beta_k \sum_{j=1}^n
\lambda_k^y(\Phi(y,\cdot)e_j)e_j.
\]
Here, $\Phi(x,y)\in\R^{n\times n}$ is a matrix and thus
$\Phi(x,y)e_j$ gives the $j$th column  of this matrix. This shows that
the expression $\lambda_k^y(\Phi(y,\cdot)e_j)$ means applying
$\lambda_k$ to the $j$th column (or row since $\Phi$ is symmetric) of
$\Phi(y,\cdot)$ with respect to $y$. Hence, we can define
$\lambda_k^y\Phi(y,\cdot)$ simply by applying 
$\lambda_k^y$ to each column/row of $\Phi(y,\cdot)$, which altogether
results into a vector. Moreover, with this definition, we can simply write
\[
\sum_{j=1}^n \lambda_k^y(\Phi(y,\cdot)e_j)e_j =
\lambda_k^y\Phi(y,\cdot)
\]

and hence
\[
s^* = \sum_{k=1}^N \beta_k\lambda_k^y\Phi(y,\cdot),
\]
which is a vector-valued function.
Finally, the coefficients $\beta_k$ are simply determined by solving
$A\beta=f$ with $A$ having entries
$\lambda_i^x\lambda_k^y\Phi(y,x)$.
\end{example}

After establishing the general theory, we will in the following
sections consider special cases. In particular, we will choose $W$ to
be the space $\R^{n\times n}$ of real-valued $n\times n$ matrices or
its subspace $\S^{n\times n}$ of symmetric matrices (Section
\ref{sec:2.3}). Then we will consider specific RKHS spaces, namely
matrix-valued Sobolev spaces $H^\sigma(\Omega;\S^{n\times n}) $ in
Section \ref{sec:2.4}, where the kernel is built from the kernel of
the corresponding real-valued Sobolev space. Finally, we will consider
functionals of the form $\lambda_k^{(i,j)} (M):=e_i^TF(M)(x_k)e_j$,
where $F\colon H^\sigma(\Omega;\S^{n\times n}) \to
H^\tau(\Omega;\S^{n\times n}) $ is a linear and bounded operator, in
particular differential operator, and derive error estimates in
Section \ref{sec:2.5}. 
In Section \ref{sec3}, a specific linear operator $F$ from dynamical
systems will be considered. 

\section{Matrix-Valued Theory} \label{sec:2.3}

We are now interested in matrix-valued functions, i.e. we set
$W=\R^{n\times n}$ or $W=\S^{n\times n}$, the space of all symmetric
$n\times n$ matrices. On $W$ we define the following inner product to make
it a Hilbert space.
\begin{eqnarray}
\langle \alpha,\beta \rangle_W &=&\sum_{i,j=1}^n \alpha_{ij}\beta_{ij},
\qquad \alpha=(\alpha_{ij}), \beta=(\beta_{ij}).
\label{W}
\end{eqnarray}

A kernel $\Phi$ is now a mapping $\Phi:\Omega\times \Omega\to
\cL(\R^{n\times n})$ and can be represented by a tensor of order $4$, i.e. we
will write
$
\Phi=(\Phi_{ijk\ell})
$
and define its action on $\alpha\in \R^{n\times n}$ by
\begin{eqnarray}
(\Phi(x,y)\alpha)_{ij} &=& \sum_{k,\ell=1}^n
\Phi(x,y)_{ijk\ell}\alpha_{k\ell}.\label{action}
\end{eqnarray}

By 2. of Lemma \ref{le:RKHS}, a necessary requirement for the kernel is the adjoint condition
$\langle\Phi(x,y)\alpha,\beta\rangle_W = \langle
\alpha,\Phi(y,x)\beta\rangle_W$, which means
\begin{eqnarray*}
\sum_{i,j=1}^n\sum_{k,\ell=1}^n\Phi(x,y)_{ijk\ell}\alpha_{k\ell}
\beta_{ij} &=& \sum_{i,j=1}^n \sum_{k,\ell=1}^n
  \Phi(y,x)_{ijk\ell} \alpha_{ij}\beta_{k\ell}\\
&=& \sum_{i,j=1}^n \sum_{k,\ell=1}^n
  \Phi(y,x)_{k\ell ij} \alpha_{k\ell}\beta_{ij}.
\end{eqnarray*}
Hence, we require
\begin{equation}\label{cond1}
\Phi(x,y)_{ijk\ell} = \Phi(y,x)_{k\ell ij}.
\end{equation}
This will motivate the choice of a kernel in \eqref{phi-tensor} in the
next section. 
The kernel $\Phi$ is positive definite, see Definition \ref{def:posdef}, if
\begin{equation}\label{posdef}
\sum_{\mu,\nu=1}^N\langle
\alpha^{(\nu)},\Phi(x_\nu,x_\mu)\alpha^{(\mu)}\rangle_W
=
\sum_{\mu,\nu=1}^N \sum_{i,j=1}^n\sum_{k,\ell=1}^n
\Phi(x_\nu,x_\mu)_{ijk\ell}\alpha^{(\nu)}_{ij}\alpha^{(\mu)}_{k\ell}
\ge 0
\end{equation}
and the sum is positive if not all of the $\alpha^{(\nu)}$ are zero.
The associated reproducing kernel Hilbert space
$\cH(\Omega;W)=\cH(\Omega;\R^{n\times n})$ consists of matrix-valued
functions.

Finally, for a given functional $\lambda\in \cH(\Omega;\R^{n\times
  n})^*$, we can write its Riesz representer as follows. Let
$E_{\mu\nu}\in\R^{n\times n}$ be the matrix with value $1$ at position
$(\mu,\nu)$ and value zero everywhere else. Then, $\{E_{\mu\nu} : 1\le
\mu,\nu\le n\}$ is an orthonormal basis of $W=\R^{n\times n}$ and the Riesz
representer of $\lambda$ hence becomes by Proposition \ref{pro:Riesz}
\[
v_\lambda(x) = \sum_{\mu,\nu=1}^n
\lambda(\Phi(\cdot,x)E_{\mu\nu})E_{\mu\nu}, \qquad x\in\Omega.
\]

In the case of the symmetric matrices, we have a similar result,
however, we need to consider a different orthonormal basis, namely
$\{E^s_{\mu\nu} : 1\le 
\mu\le \nu\le n\}$.
We define $E^s_{\mu\mu}$ to be the matrix with value 1 at position
$(\mu,\mu)$ and value zero everywhere else. For $\mu<\nu$, we define
$E^s_{\mu\nu}$ to be the matrix with value $1/\sqrt{2}$ at positions
$(\mu,\nu)$ and $(\nu,\mu)$ and value zero everywhere else. 
It is easy to see that $\{E^s_{\mu\nu} : 1\le
\mu\le \nu\le n\}$ is an orthonormal basis of $W=\S^{n\times n}$.

For a given functional $\lambda\in \cH(\Omega;\S^{n\times
  n})^*$,  the Riesz
representer of $\lambda$ hence is by Proposition \ref{pro:Riesz}
\begin{eqnarray}
v_\lambda(x) &=& \sum_{1\le \mu\le\nu\le n}
\lambda(\Phi(\cdot,x)E^s_{\mu\nu})E^s_{\mu\nu}, \qquad x\in\Omega.
\label{Rieszr}
\end{eqnarray}
%

\section{Matrix-Valued Sobolev Spaces}
\label{sec:2.4}

In the following, we will be concerned with specific functionals
defined on specific reproducing kernel Hilbert spaces. We start with
discussing the spaces.

Throughout this paper, we will assume that $H^\sigma(\Omega)$ denotes
the Sobolev space of order $\sigma>d/2$, where the weak derivatives
are measured in the $L_2(\Omega)$-norm. However, $\sigma$ does not
necessarily have to be an integer and the space can then be defined,
for example, by interpolation. We will always assume that
$\sigma>d/2$ such that the Sobolev embedding theorem yields
$H^\sigma(\Omega)\subseteq C(\Omega)$ which particularly means that
$H^\sigma(\Omega)$ has a reproducing kernel.  The kernel is uniquely
determined by the inner product, but different equivalent inner
products allow us to choose different kernels. Examples of such
kernels comprise of the Sobolev (or Matern) kernels and Wendland's
radial basis functions
(see \cite{Fasshauer-07-1,Wendland-95-1,Schaback-11-1}). 

We will also assume that
$\Omega\subseteq\R^d$ is a bounded domain with a boundary which is at
least Lipschitz continuous.

\begin{definition}\label{def:Sob}
Let $\Omega\subseteq\R^d$ and $\sigma>d/2$ be given. Then, the
matrix-valued Sobolev space $H^\sigma(\Omega;\R^{n\times n})$ consists of
all matrix-valued functions $M$ having each component $M_{ij}$ in
$H^\sigma(\Omega)$.
Similarly, the
Sobolev space $H^\sigma(\Omega;\S^{n\times n})$ consists of
all symmetric matrix-valued functions $M$ having each component $M_{ij}$ in
$H^\sigma(\Omega)$.
\end{definition}

$H^\sigma(\Omega;\R^{n\times n})$ and $H^\sigma(\Omega;\S^{n\times n})$ are Hilbert spaces with inner
product given by
\[
\langle M, S\rangle_{H^\sigma(\Omega;\R^{n\times n})}:=\sum_{i,j=1}^n \langle M_{ij},S_{ij}\rangle_{H^\sigma(\Omega)};
\]
the same inner product can be used for $H^\sigma(\Omega;\S^{n\times
  n})$. They are also reproducing kernel Hilbert spaces. 

\begin{lemma} \label{kernel1} Let $\Omega\subseteq\R^d$ and $\sigma>d/2$ be
  given. Assume that $\phi:\Omega\times\Omega\to\R$ is a reproducing
  kernel of $H^\sigma(\Omega)$. Then, $H^\sigma(\Omega;\R^{n\times n})$ and $H^\sigma(\Omega;\S^{n\times n})$ are also reproducing kernel Hilbert spaces with reproducing kernel
  $\Phi$ defined by
\begin{equation}\label{phi-tensor}
\Phi(x,y)_{ijk\ell}:=\phi(x,y)\delta_{ik}\delta_{j\ell}
\end{equation}
for $x,y\in\Omega$ and $1\le i,j,k,\ell\le n$.
\end{lemma}

\begin{proof}
We have to verify the two defining properties of a reproducing
kernel, see Definition \ref{def:RKHS}. First of all, we obviously have
$\Phi(\cdot,x)\alpha \in H^\sigma(\Omega;\R^{n\times n})$ for all
$x\in\Omega$ and all $\alpha\in\R^{n\times n}$ since
\[
(\Phi(\cdot,x)\alpha)_{ij} = \sum_{k,\ell=1}^n
\Phi(\cdot,x)_{ijk\ell}\alpha_{k\ell}
= \sum_{k,\ell=1}^n
\phi(\cdot,x)\delta_{ik}\delta_{j\ell}\alpha_{k\ell} =
\phi(\cdot,x)\alpha_{ij}
\]
and $\phi$ is a reproducing kernel of $H^\sigma(\Omega)$.
For $H^\sigma(\Omega;\S^{n\times n})$,  note that $\Phi(\cdot,x)\alpha$ is symmetric if $\alpha$ is symmetric.

Secondly, we have the reproduction property. If once again
$\alpha\in\R^{n\times n}$ and $f\in H^\sigma(\Omega;\R^{n\times n})$
then the computation just made  shows
\begin{eqnarray*}
\langle f,\Phi(\cdot,x)\alpha \rangle_{H^\sigma(\Omega;\R^{n\times n})}
& = & \sum_{i,j=1}^n \langle f_{ij},
  (\Phi(\cdot,x)\alpha)_{ij}\rangle_{H^\sigma(\Omega)} \\
& = & \sum_{i,j=1}^n \langle f_{ij},\phi(\cdot,x)\alpha_{ij}\rangle_{H^\sigma(\Omega)} \\
& = & \sum_{i,j=1}^n \alpha_{ij}f_{ij}(x) = \langle
f(x),\alpha\rangle_{\R^{n\times n}},
\end{eqnarray*}
using the reproduction property of $\phi$ in $H^\sigma(\Omega)$. The
proof for $H^\sigma(\Omega;\S^{n\times n})$ is the same. 
\end{proof}

\begin{corollary} Let the assumptions of Lemma \ref{kernel1} hold with 
a positive definite kernel $\phi:\Omega\times\Omega\to\R$. Then, also the
matrix-valued kernel $\Phi$ is positive definite.
\end{corollary}

\begin{proof}
The kernel is positive definite in the sense of (\ref{posdef}), since
we have
\begin{eqnarray*}
\sum_{\mu,\nu=1}^N \sum_{i,j=1}^n\sum_{k,\ell=1}^n
\Phi(x_\nu,x_\mu)_{ijk\ell}\alpha^{(\nu)}_{ij}\alpha^{(\mu)}_{k\ell} &
=&
\sum_{\mu,\nu=1}^N \sum_{i,j=1}^n\sum_{k,\ell=1}^n
\phi(x_\nu,x_\mu)\delta_{ik}\delta_{j\ell}\alpha^{(\nu)}_{ij}\alpha^{(\mu)}_{k\ell}\\
& = &
\sum_{i,j=1}^n\sum_{\mu,\nu=1}^N
\phi(x_\nu,x_\mu)\alpha_{ij}^{(\nu)}\alpha_{ij}^{(\mu)}\ge 0
\end{eqnarray*}
and at least one of the inner sums is positive.
\end{proof}

Next, we will discuss the functionals on $H^\sigma(\Omega;\R^{n\times
  n})$ and $H^\sigma(\Omega;\S^{n\times
  n})$ that we are interested in. Note that using a kernel of the form
(\ref{phi-tensor}) together with point evaluations would simply lead to
a component-wise treatment. Hence, in this situation, dealing with
each component separately would be more efficient.

Here, however, we are interested in the following situation. Suppose
$F:H^\sigma(\Omega;\R^{n\times n})\to H^\tau(\Omega;\R^{n\times n})$ (or
$F:H^\sigma(\Omega;\S^{n\times n})\to H^\tau(\Omega;\S^{n\times n})$) is a
linear and bounded map, i.e. there is a constant $C>0$ such that
\[
\|F(M)\|_{H^\tau(\Omega;\R^{n\times n})} \le C \|M\|_{H^\sigma(\Omega;\R^{n\times n})}, \qquad M\in
H^\sigma(\Omega;\R^{n\times n}).
\]
Suppose further that $\tau>d/2$ so that $F(M)\in C(\Omega;\R^{n\times n})$ is continuous. Then, we can define functionals of the form
\[
\lambda_k^{(i,j)}(M)=e_i^TF(M)(x_k)e_j
\]
for $1\le i,j\le n$ (or $1\le i\le j\le n$) and $1\le k\le N$, where
$X=\{x_1,\ldots,x_N\}$ is a given discrete point set in $\Omega$. 

We will specify the mapping $F$ later on but we can derive a general
theory using just these assumptions. 

\section{Error Analysis}\label{sec:2.5}
In this section we will start with analysing the reconstruction
error. Here, we will follow general ideas from scattered data
approximation. In particular, we will measure the error in terms of
the so-called fill distance or mesh norm
\[
h_{X,\Omega}:=\sup_{x\in\Omega}\min_{x_i \in X}\|x-x_i\|_2.
\]

This means that we can derive the classical error estimates based upon
sampling inequalities also in this case. We will require the following
result (see \cite{Narcowich-etal-05-1}).

\begin{lemma}\label{zeros}
Let $\Omega\subseteq\R^d$ be a bounded domain with Lipschitz
continuous boundary. Let $\sigma>d/2$ and let
$X=\{x_1,\ldots,x_N\}\subseteq\Omega$. If $f\in H^\sigma(\Omega)$
vanishes on $X$, then there is a constant $C>0$ independent of $X$ and
$f$ such that
\[
\|f\|_{L_\infty(\Omega)} \le Ch_{X,\Omega}^{\sigma-d/2}
\|f\|_{H^{\sigma}(\Omega)}.
\]
\end{lemma}

We can now use this result component-wise to derive estimates for the
matrix-valued set-up. We will do this immediately for the situation we
are interested in, which gives our first main result of this paper.

\begin{theorem}\label{th:error}
Let $\Omega\subseteq\R^d$ be a bounded domain with Lipschitz continuous
boundary. Let $\sigma,\tau >d/2$ be given and let
$F:H^{\sigma}(\Omega;\R^{n\times n})\to H^\tau(\Omega;\R^{n\times n})$
($F:H^{\sigma}(\Omega;\S^{n\times n})\to H^\tau(\Omega;\S^{n\times n})$) be
linear and bounded.  Finally, let
$X=\{x_1,\ldots,x_N\}\subseteq\Omega$ be given and let
\[
\lambda_k^{(i,j)} (M):=e_i^TF(M)(x_k)e_j, \qquad 1\le k\le N, \quad 1\le
i,j\le n \quad (1\le
i\le j\le n).
\]
Then each
 $\lambda_k^{(i,j)}$ belongs to the dual of
$H^\sigma(\Omega;\R^{n\times n})$ ($H^\sigma(\Omega;\S^{n\times n})$).

Let us further assume that they are linearly independent.
If $S$ denotes the optimal recovery of $M\in
H^\sigma(\Omega;\R^{n\times n})$ ($H^\sigma(\Omega;\S^{n\times n})$)
in the sense of Definition \ref{optrec} using these functionals 
and a reproducing kernel of $H^\sigma(\Omega;\R^{n\times n})$
($H^\sigma(\Omega;\S^{n\times n})$), 
then
\[
\|F(M)-F(S)\|_{L_\infty(\Omega;\R^{n\times n})} \le C h_{X,\Omega}^{\tau-d/2}
\|M\|_{H^\sigma(\Omega;\R^{n\times n})},
\]
where $\|A\|_{L_\infty(\Omega;\R^{n\times
    n})}=\max_{i,j=1,\ldots,n}\|a_{ij}(x)\|_{L_\infty(\Omega)}$. 
\end{theorem}
\begin{proof} We only consider the case $\R^{n\times n}$ as the proof for $\S^{n\times n}$ is similar. Obviously, the $\lambda_k^{(i,j)}$ are linear. Because
  of our assumptions, $F(M)$ is indeed continuous by the Sobolev
  embedding theorem, i.e. the functionals are
  well-defined. Furthermore, 
\[
|\lambda_k^{(i,j)}(M)|\le C\|F(M)\|_{H^\tau(\Omega;\R^{n\times n})} \le C
\|M\|_{H^\sigma(\Omega;\R^{n\times n}) }, \qquad M\in
H^{\sigma}(\Omega;\R^{n\times n}),
\]
by the Sobolev embedding theorem and by the continuity of $F$. This
means that all functionals indeed belong to the dual of
$H^\sigma(\Omega;\R^{n\times n})$.
%

For the error estimate we note that  the matrix-valued function
$F(M)-F(S)\in H^\tau(\Omega;\R^{n\times n})$ vanishes on the data set
$X$. Hence, we can apply Lemma
\ref{zeros} to each component of $F(M)-F(S)$ yielding
\begin{eqnarray*}
\|F(M)-F(S)\|_{L_\infty(\Omega;\R^{n\times n})} &\le& C
h_{X,\Omega}^{\tau-d/2}\|F(M-S)\|_{H^{\tau}(\Omega;\R^{n\times n})}\\
&\le & C h_{X,\Omega}^{\tau-d/2}\|M-S\|_{H^\sigma(\Omega;\R^{n\times n})}\\
& \le & C h_{X,\Omega}^{\tau-d/2}\|M\|_{H^{\sigma}(\Omega;\R^{n\times n})},
\end{eqnarray*}
using also the continuity of $F$ and the fact that $S$ is the
$H^\sigma(\Omega;\R^{n\times n})$ optimal recovery of $M$.
\end{proof}

To show linear independence, we follow the scalar-valued case
\cite{Giesl-Wendland-07-1} and define singular points for a general linear
differential operator $F$, mapping matrix-valued functions to
matrix-valued functions. We will then apply the 
rather general result of Theorem \ref{th:error} to a particular class
of operators $F$.

%

\begin{definition} \label{def:operator}
Let $n,d\in \N$, $\Omega\subseteq\R^d$, $\sigma>m+d/2$ and
$\tau=\sigma-m$. Let $W=\R^{n\times n}$ or $W=\S^{n\times n}$.
Let $F:H^{\sigma}(\Omega;W)\to H^\tau(\Omega;W)$
be a differential operator of degree $m$ of the form 
\[
F(M)(x)=\sum_{|\alpha|\le m} c_{\alpha}(x) [D^\alpha M(x)]
\]
where $D^\alpha$ is applied component-wise and
$c_\alpha:\Omega\to\cL(W)$ is of such a form that $x\mapsto 
c_\alpha(x)[D^\alpha M(x)] \in H^\tau(\Omega; W)$ for every $M\in
H^\sigma(\Omega;W)$.  

We define $x$ to be a {\em singular point} of $F$ if  for all $|\alpha|\le
m$ the linear map $c_\alpha(x)$ is not invertible. 
\end{definition}

In the next lemma we will show symmetry properties for $F$, defined on the
symmetric matrices, which will later be needed for explicit
calculations. 

\begin{lemma} \label{le}
Assume that 
$F:H^{\sigma}(\Omega;\S^{n\times n})\to H^\tau(\Omega;\S^{n\times
  n})$ is a differential operator as in
Definition \ref{def:operator}, i.e. in particular $c_\alpha(x)(M)\in \S^{n\times
  n}$ for $M\in\S^{n\times n}$. Assume furthermore that  the kernel
$\Phi(x,y)_{ijk\ell}=\phi(x,y)\delta_{ik}\delta_{j\ell}$ from
\eqref{phi-tensor} is used. Then 
\begin{equation}
F(\Phi(\cdot,x)_{\cdot,\cdot,\mu,\nu})_{ij}
=
F(\Phi(\cdot,x)_{\cdot,\cdot,\nu,\mu})_{ji}.\label{F:symm}
\end{equation}
\end{lemma}
\begin{proof}
The linear map $c_\alpha (x)$ can,
similar to \eqref{action}, be described by a tensor of order $4$, i.e. 
\begin{equation}\label{c-rep}
(c_\alpha(x)(M))_{ij}=\sum_{k,\ell=1}^n c_\alpha(x)_{ijk\ell} M_{k\ell}.
\end{equation}
We show that we can assume
\begin{equation}c_\alpha(x)_{ijk\ell}=c_\alpha(x)_{ij\ell k}\label{c-prop}
\end{equation}
for all $x\in \Omega$ without loss of generality. 
Indeed, let $c_\alpha$ be given satisfying \eqref{c-rep} and define $\tilde{c}_\alpha$ by
$$\tilde{c}_\alpha(x)_{ijk\ell}:=\tilde{c}_\alpha(x)_{ij\ell k}:=
\frac{1}{2}\left(c_\alpha(x)_{ijk\ell}+c_\alpha(x)_{ij\ell k}\right).$$
It is clear that $\tilde{c}$ satisfies \eqref{c-prop} and we also
have, using $M\in \S^{n\times n}$, 
\begin{eqnarray*}
\sum_{k,\ell=1}^n \tilde{c}_\alpha(x)_{ijk\ell} M_{k\ell}
&=&
\sum_{k=1}^n \tilde{c}_\alpha(x)_{ijkk} M_{kk}
+
\sum_{1\le k<\ell\le n} \tilde{c}_\alpha(x)_{ijk\ell} [M_{k\ell}+M_{\ell k}]\\
&=&
\sum_{k=1}^n c_\alpha(x)_{ijkk} M_{kk}
+2\sum_{1\le k<\ell\le n} \tilde{c}_\alpha(x)_{ijk\ell} M_{k\ell}\\
&=&
\sum_{k=1}^n c_\alpha(x)_{ijkk} M_{kk}
+\sum_{1\le k<\ell\le n} 
\left(c_\alpha(x)_{ijk\ell}+c_\alpha(x)_{ij\ell k}\right) M_{k\ell}\\
&=&
\sum_{k,\ell=1}^n c_\alpha(x)_{ijk\ell} M_{k\ell}
=
(c_\alpha(x)(M))_{ij}.
\end{eqnarray*}
For  $M\in \S^{n\times n}$ we have $c_\alpha(x)(M)\in \S^{n\times
  n}$ and hence

\begin{eqnarray*}
\sum_{k,\ell=1}^n c_\alpha(x)_{ijk\ell} M_{k\ell}
&=&(c_\alpha(x)(M))_{ij}
=(c_\alpha(x)(M))_{ji}
=\sum_{k,\ell=1}^n c_\alpha(x)_{jik\ell} M_{k\ell}\\
&=&\sum_{k,\ell=1}^n c_\alpha(x)_{jik\ell} M_{\ell k}
=\sum_{k,\ell=1}^n c_\alpha(x)_{ji\ell k} M_{k\ell}
\end{eqnarray*}
as $M\in \S^{n\times n}$. 
Choosing $M=E_{\mu\nu}^s$ to be a basis
``vector'' of $\S^{n\times  n}$ shows, using \eqref{c-prop},
\begin{eqnarray*}
\sum_{k,\ell=1}^n c_\alpha(x)_{ijk\ell}(E_{\mu\nu}^s)_{k\ell} &=&
\frac{1}{\sqrt{2}} \left[
  c_\alpha(x)_{ij\mu\nu}+c_\alpha(x)_{ij\nu\mu}\right]
  =\sqrt{2}c_\alpha(x)_{ij\mu\nu},\\
\sum_{k,\ell=1}^n c_\alpha(x)_{ji\ell k}(E_{\mu\nu}^s)_{k\ell} &=&
\frac{1}{\sqrt{2}} \left[
  c_\alpha(x)_{ji\nu\mu}+c_\alpha(x)_{ji\mu\nu}\right]
  =\sqrt{2}c_\alpha(x)_{ji\nu\mu},
\end{eqnarray*}
i.e. 
\begin{equation}\label{c:symm}
c_\alpha(x)_{ijk\ell} = c_\alpha(x)_{ji\ell k}.
\end{equation}
For \eqref{F:symm} note that
\[
D^\alpha \Phi(\cdot,x)_{i,j,\mu,\nu}=
D^\alpha \phi(\cdot,x) \delta_{i\mu}\delta_{j\nu}
\]
so that
\begin{eqnarray*}
F(\Phi(\cdot,x)_{\cdot,\cdot,\mu,\nu})_{ij}
&=&
\sum_{|\alpha|\le m} D^\alpha \phi(\cdot,x)\sum_{k,\ell=1}^nc_{\alpha}(x)_{ijk\ell}
\delta_{k\mu}\delta_{\ell\nu}
=
\sum_{|\alpha|\le m} D^\alpha \phi(\cdot,x)c_{\alpha}(x)_{ij\mu\nu} \\
&=&
\sum_{|\alpha|\le m} D^\alpha \phi(\cdot,x)c_{\alpha}(x)_{ji\nu\mu} 
=\sum_{|\alpha|\le m} D^\alpha \phi(\cdot,x)\sum_{k,\ell=1}^nc_{\alpha}(x)_{jik\ell}
\delta_{k\nu}\delta_{\ell\mu}\\
&=&
F(\Phi(\cdot,x)_{\cdot,\cdot,\nu,\mu})_{ji},
\end{eqnarray*}
where we have used  \eqref{c:symm}.
\end{proof}

\begin{proposition}\label{pro:independent}
Let $\sigma>m+d/2$ and $F$ be a linear differential operator
$F:H^{\sigma}(\Omega;\R^{n\times n})\to H^\tau(\Omega;\R^{n\times n})$
($F:H^{\sigma}(\Omega;\S^{n\times n})\to H^\tau(\Omega;\S^{n\times
  n})$) as in Definition \ref{def:operator}. Let $X=\{x_1,\ldots,x_N\}$ be a
set of pairwise 
distinct points which are not singular points of $F$. 

Then the functionals
\[
\lambda_k^{(i,j)} (M):=e_i^TF(M)(x_k)e_j, \qquad 1\le k\le N, 1\le
i,j\le n \quad (1\le i\le j\le n).
\]
are bounded and linearly independent over
$H^{\sigma}(\Omega;\R^{n\times n})$ ($H^{\sigma}(\Omega;\S^{n\times
  n})$). 
\end{proposition}
\begin{proof}
The boundedness of the functionals is clear from the assumptions.

%
%
%

We will prove the linear independence of the functionals over
$H^{\sigma}(\Omega;\S^{n\times n})$.  In Theorem \ref{th:error}, we
have already seen  that the functionals belong to the dual of
$H^{\sigma}(\Omega;\S^{n\times n})$. 

Now assume that
\[
\sum_{k=1}^N\sum_{1\le i\le j\le n}d_k^{(i,j)} \lambda_k^{(i,j)} =0
\]
on $H^{\sigma}(\Omega;\S^{n\times n})$ with certain coefficients
$d_k^{(i,j)}$. We need to show that all $d_k^{(i,j)}=0$. 

To this end, let $g\in C_0^\infty(\R^d;\R)$ be a flat bump function,
i.e. a nonnegative, compactly supported function with support
$B(0,1)$, satisfying $g(x)=1$ on $B(0,1/2)$. 
%
%

Fix $1\le \ell \le N$, as well as $i^*,j^*\in \{1,\ldots,n\}$ with
$i^*\le j^*$. Since
$x_\ell$ is no singular point of $F$ there exists a minimal
$|\beta|\le m$ such that $c_{\beta}(x_\ell)$ is invertible. The
function 
\[
g_\ell(x)=\frac{1}{\beta !}(x-x_\ell)^\beta
g\left(\frac{x-x_\ell}{q_X}\right),
\]
where $q_X$ denotes the separation distance of $X$, then satisfies 
$D^\alpha g_\ell(x_k)=0$ for all $|\alpha|\le m$ and $x_k\not=x_\ell$.
Moreover, $D^\alpha g_\ell(x_\ell)=0$ for $\alpha\not=\beta$ and $D^\beta g_\ell(x_\ell)=1$. Hence, defining the matrix valued function
$G\in  H^\sigma(\Omega;\S^{n\times n})$ by 
$G(x)=g_\ell(x) c_\beta (x_\ell)^{-1} E_{i^* j^*}^s$,
we have
\begin{eqnarray*}
0&=&\sum_{k=1}^N\sum_{1\le i\le j\le n}d_k^{(i,j)} \lambda_k^{(i,j)} (G)\\
&=&\sum_{k=1}^N\sum_{1\le i\le j\le n}d_k^{(i,j)} e_i^TF(G)(x_k)e_j\\
&=&\sum_{k=1}^N\sum_{|\alpha|\le m}\sum_{1\le i\le j\le n}
d_k^{(i,j)} e_i^T c_\alpha(x_k)c_\beta(x_\ell)^{-1}E_{i^*j^*}^s e_j\ D^\alpha g_\ell(x_k)\\
&=&\sum_{1\le i\le j\le n}
d_\ell^{(i,j)} e_i^T c_\beta(x_\ell)c_\beta(x_\ell)^{-1}E_{i^*j^*}^se_j \\
&=&c_{i^*,j^*}d_\ell^{(i^*,j^*)},
\end{eqnarray*}
where $c_{i^*,j^*}=\frac{1}{\sqrt{2}}$ for $i^*\not=j^*$ and $c_{i^*,i^*}=1$.
Since $\ell, i^*,j^*$ were chosen arbitrarily, this shows the linear
independence. 
\end{proof}

Now we consider a special type of $F$, which will later arise in the
application within Dynamical Systems. 

\begin{theorem}\label{th:errordynsys}
Let $\Omega\subseteq\R^d$ be a bounded domain with Lipschitz continous
boundary. Let $\sigma>d/2+1$ and let $V\in H^{\sigma-1}(\Omega;\R^{n\times n})$
and $f\in H^{\sigma-1}(\Omega;\R^n)$.
Define $F:H^{\sigma}(\Omega;\S^{n\times n})\to H^{\sigma-1}(\Omega;\S^{n\times n})$ by
\[
F(M)(x):=V(x)^TM(x)+M(x)V(x)+M'(x),
\]
where $(M'(x))_{ij}=\nabla M_{ij}(x)\cdot f(x)$.

For each $x_0\in \Omega$ with $f(x_0)=0$ (equilibrium point), we
assume that all eigenvalues of $V(x_0)$ have negative real part
(positive real part). 

Finally, let $X=\{x_1,\ldots,x_N\}\subseteq\Omega$ be a set of pairwise distinct points and let
\[
\lambda_k^{(i,j)} (M):=e_i^TF(M)(x_k)e_j, \qquad 1\le k\le N,\quad  1\le
i\le j\le n.
\]
Then, each $\lambda_k^{(i,j)}$ belongs to the dual of
$H^\sigma(\Omega;\S^{n\times n})$ and they are linearly independent.  If
$S$ denotes the optimal recovery of $M\in H^\sigma(\Omega;\S^{n\times
  n})$ in the sense of Definition \ref{optrec} using these functionals,
then
\[
\|F(M)-F(S)\|_{L_\infty(\Omega;\S^{n\times n})} \le C h_{X,\Omega}^{\sigma-1-n/2}
\|M\|_{H^\sigma(\Omega;\S^{n\times n})}.
\]
\end{theorem}
\begin{proof}
The operator $F$ is a differential operator of degree 1 as in Definition \ref{def:operator} with
\begin{eqnarray*}
 c_0(x) (M)&=&V(x)^TM+MV(x)\\
 c_{e_i}(x)(M)&=&f_i(x) M
 \end{eqnarray*}
We have $x\mapsto c_\alpha(x)[D^\alpha M(x)]\in
H^{\sigma-1}(\Omega;\S^{n\times n})$ for every $M\in H^\sigma
(\Omega;\S^{n\times n})$.
  To apply Proposition \ref{pro:independent}, we have to show that there are no singular points in $\Omega$.

Case 1: If $f(x)\not=0$, then there is an $i^*\in\{1,\ldots,n\}$ with $f_{i^*}(x)\not=0$ and hence
$c_{e_{i^*}}(x)$ is invertible with $c_{e_{i^*}}(x)^{-1}=\frac{1}{f_i(x)}\mbox{id}$.

Case 2: If $f(x)=0$, then by assumption $V(x)$ ($-V(x)$) has eigenvalues with only negative real part. Then the so-called Lyapunov equation
\[
V(x)^TM+MV(x)=C \qquad (-C)
\]
has a unique solution for every $C\in \S^{n\times n}$ , see e.g. \cite[Theorem 4.6]{Khalil-92-1}, i.e. the operator $c_0(x)$ is injective and, because it maps the finite-dimensional space $\S^{n\times n}$ into itself, also bijective.

The rest follows from the previous results, in particular Theorem
\ref{th:error} by setting $\tau=\sigma-1$.
\end{proof}


\section{Contraction metric}
\label{sec3}

In this section we will apply the previous general results to the
ODE problem of constructing a contraction metric mentioned in the
introduction. Hence, we study the autonomous ODE
\begin{eqnarray}
\dot{\bx}&=&\bff(\bx)\label{ODE}
\end{eqnarray}
where $\bff\in C^1(\R^n,\R^n)$; further assumptions on the smoothness
of $\bff$ will be made later. The solution $\bx(t)$ with initial
condition $\bx(0)=\xi$ is denoted by $\bx(t)=:S_t\xi$ and is assumed
to exist for all $t\ge 0$. 

We are interested in the existence, uniqueness and exponential
stability of an equilibrium, as well as the determination of its basin
of attraction. An equilibrium is a point $\bx_0\in \R^n$ such that
$f(x_0)=0$. Its basin of attraction is defined by 
$A(x_0)=\{x\in \R^n\mid \lim_{t\to\infty} S_tx=x_0\}.$

If the equilibrium is known, then Lyapunov functions are one way of
analysing the basin of attraction of the equilibrium as well as its
basin of attraction, see the recent survey article
\cite{Giesl-Hafstein-15-2} for constructing such Lyapunov functions. A
different way of studying stability and the 
basin of attraction, which does not require any knowledge about the
equilibrium and which is also robust with respect to perturbations of
the ODE uses contraction metrics. A Riemannian contraction metric is a
matrix-valued function $M\colon \R^n\to \S^{n\times n}$, such that
$M(x)$ is positive definite for every $x$. It defines a
(point-dependent) scalar product on $\R^n$ by $\langle
v,w\rangle_M=v^TM(x)w$. 
For $M$ to be a contraction metric, we require the distance between
adjacent solutions to decrease with respect to such a contraction
metric. This can be expressed by the negative definiteness of $F(M)(x)
=Df^T(x)M(x)+M(x)Df(x)+M'(x)$,
see \eqref{pde-const} and Theorem \ref{th:1}.

Contraction analysis can be used to study the distance between
trajectories, without reference to an attractor, establishing
(exponential) attraction of adjacent trajectories, see
\cite{Krasovskii-59-1,Hahn-67-1}, see also \cite[Section
  2.10]{Giesl-Hafstein-15-1}; it can be generalised to the study of a
Finsler-Lyapunov function \cite{Forni-Sepulchre-14-1}. 

If contraction to a trajectory through $\bx$ occurs with respect to
all adjacent trajectories, then solutions converge to an
equilibrium. If the attractor is, e.g., a periodic orbit, then
contraction cannot occur in the direction tangential to the
trajectories. Hence, contraction analysis for periodic orbits assumes
contraction only to occur in a suitable $(n-1)$-dimensional subspace
of the tangent space.  
Contraction metrics for periodic orbits have been studied by Borg
\cite{Borg-60-1} with the Euclidean metric and Stenstr\"om
\cite{Stenstrom-62-1} with a general Riemannian metric. 
Further results using a contraction metric to establish existence,
uniqueness, stability and information about the basin of attraction of
a periodic orbit have been obtained in
\cite{Hartman-64-1,Leonov-etal-96-1}.  

Only few converse theorems for contraction metrics have been obtained,
establishing the existence of a contraction metric, see
\cite{Giesl-15-1} for some references. 
 Constructive converse theorems, providing algorithms for the explicit
 construction of a contraction metric, are given in
 \cite{Aylward-etal-08-1} for the global stability of an equilibrium
 in  polynomial systems, using Linear Matrix Inequalities (LMI) and
 sums of squares  (SOS). An algorithm to construct a continuous
 piecewise affine (CPA) contraction metric for periodic orbits in
 time-periodic systems using semi-definite optimization has been
 proposed in \cite{Giesl-Hafstein-13-1}.

In \cite{Giesl-15-1}, the existence of a contraction metric for an
equilibrium was shown which satisfies $F(M)=-C$, where $C$ is a given
symmetric positive definite matrix. In  
\cite{Giesl-Wendland-??-1}, summarised in the following theorem, we
establish existence and uniqueness of solutions of the more general
matrix-valued PDE \eqref{matrixeq3}. 

\begin{theorem}\label{est}
Let $f\in C^\sigma(\R^n,\R^n)$, $\sigma\ge 2$. Let $x_0$ be an
exponentially stable equilibrium of $\dot{x}=f(x)$ with basin of
attraction $A(x_0)$.  Let $C_i\in C^{\sigma-1}(A(\bx_0),\S^{n\times
  n})$, $i=1,2$, such that $C_i(\bx)$ is a positive definite matrix
for all $\bx\in A(\bx_0)$. 
Then,  for $i=1,2$ the matrix equation
\begin{equation}\label{matrixeq3}
Df^T(x)M_i(x)+M_i(x)Df(x)+M_i'(x)=-C_i(x)
\end{equation}
has a unique solution $M_i\in C^{\sigma-1}(A(x_0), \S^{n\times n})$.

Let $K\subseteq A(\bx_0)$ be a compact set. Then there is a constant
$c$, independent of $M_i$ and $C_i$ such that 
\[
\|M_1-M_2\|_{L_\infty(K)}\le c \|C_1-C_2\|_{L_\infty(\overline{\gamma^+(K)})}
\]
where $\gamma^+(K)=\bigcup_{t\ge 0}S_tK$.
\end{theorem}

The theorem shows that if $\|F(M)(x)-F(S)(x)\|\le \epsilon$ for all
$x\in \overline{\gamma^+(K)}$, then $\|M(x)-S(x)\|\le c\epsilon$ for
all $x\in K$. In particular, as $M$ is positive definite in $K$, so is
$S$, if $\epsilon$ is small enough. 

Note that for a positively invariant and compact set $K$ we have
$\overline{\gamma^+(K)}=K$. 



Let $f\in C^\sigma(\R^n,\R^n)$, $\sigma\ge 2$. 
In what follows, we will
always have $d=n$.  Let $x_0$ be an
exponentially stable equilibrium of 
$\dot{x}=f(x)$ with basin of attraction $A(x_0)$.  

Then, our strategy for constructing a Riemannian contraction metric
is to choose a symmetric and positive definite matrix $C\in\S^{n\times
  n}$ and to approximate the partial differential equation
\begin{equation}\label{pde-const}
F(M)(x):=Df^T(x)M(x)+M(x)Df(x)+\nabla M(x)\cdot f(x) = - C.
\end{equation}
using generalised collocation as described in the previous
sections. Here we have used the simplified notation $\nabla M(x)\cdot
f(x)$ to denote the $n\times n$ matrix with entries $\nabla
M_{ij}(x)^Tf(x)$. This can be summarised as follows. We set
$W=\S^{n\times n}$ to be 
the space of all symmetric $n\times n$ matrices with inner product as
in \eqref{W}. Furthermore, we define $\cH=H^\sigma(\Omega;W)$
to be the matrix-valued Sobolev space of Definition \ref{def:Sob}
with reproducing kernel  $\Phi:\Omega\times\Omega\to\cL(W)$
as in \eqref{phi-tensor}, where $\Omega\subseteq
\R^n$ will be chosen appropriately later on.
We then define the linear functionals
$\lambda_k^{(i,j)}:H^\sigma(\Omega;W)\to\R$ by
\begin{eqnarray}
\lambda_{k}^{(i,j)}(M) &=& e_i^T \left[Df^T(x_k)M(x_k)
  +M(x_k)Df(x_k)+\nabla M(x_k) \cdot f(x_k)\right]e_j\label{functionals}\\
&=:& e_i^T F_k(M)e_j\nonumber\\
&=& e_i^T F(M)(x_k)e_j\nonumber
\end{eqnarray}
for $x_k\in\Omega$, $1\le k\le N$ and $1\le i\le j\le n$. Here, $e_i$
denotes once again the $i$th unit vector in $\R^n$.

Then, we can compute the solution $S$ of the optimal recovery problem
as in Definition \ref{optrec}. This gives the following result.

\begin{theorem}\label{th:S} 
Let $\sigma>n/2+1$ and let $\Phi:\Omega\times\Omega\to\cL(\S^{n\times
  n})$ be a reproducing kernel of $H^\sigma(\Omega;\S^{n\times n})$.
Let $X=\{x_1,\ldots,x_N\}\subseteq\Omega$ be pairwise distinct points and let
$\lambda_k^{(i,j)}\in H^\sigma(\Omega;\S^{n\times n})^*$, $1\le k\le N$ and $1\le
i,j\le n$ be defined by (\ref{functionals}) with $V:=Df$ satisfying the
conditions of Theorem \ref{th:errordynsys}.
Then there is a unique  function $S\in H^\sigma(\Omega;\S^{n\times n})$ solving
\[
\min\left\{ \|M\|_{\cH} : \lambda_{k}^{(i,j)}(M) = -C_{ij}, 1\le
i\le j\le n, 1\le k\le N\right\},
\]
where $C=(C_{ij})_{i,j=1,\ldots,n}$ is a symmetric, positive definite matrix.

It has the form
\begin{eqnarray}
S(x) &=& \sum_{k=1}^N\sum_{1\le i\le j\le n} \gamma_k^{(i,j)}
\sum_{1\le \mu\le \nu\le n}\lambda_{k}^{(i,j)}(\Phi(\cdot,x)E^s_{\mu\nu})E^s_{\mu\nu} \nonumber\\
&=& \sum_{k=1}^N\sum_{1\le i\le j\le n}
\gamma_k^{(i,j)}
\bigg[
\sum_{ \mu=1}^n
F_k(\Phi(\cdot,x)_{\cdot,\cdot,\mu,\mu})_{ij}E_{\mu\mu}\nonumber\\
&&+
\frac{1}{2}
\sum_{\substack{ \mu,\nu=1\\\mu\not =\nu}}^n
[F_k(\Phi(\cdot,x)_{\cdot,\cdot,\mu,\nu})_{ij}+
F_k(\Phi(\cdot,x)_{\cdot,\cdot,\nu,\mu})_{ij}]E_{\mu\nu}\bigg],\label{form1}
\end{eqnarray}
where the coefficients $\gamma_k=(\gamma_k^{(i,j)})_{1\le i\le j\le n}$
are determined by $\lambda_\ell^{(i,j)}(S)=-C_{ij}$ for $1\le i\le j\le n$.

If the kernel $\Phi$ is given by \eqref{phi-tensor} then we also have the alternative expression
\begin{eqnarray}
S(x) &=&
\sum_{k=1}^N\sum_{i,j=1}^n\beta_k^{(i,j)} \sum_{\mu,\nu=1}^n
F_k(\Phi(\cdot,x)_{\cdot,\cdot,\mu,\nu})_{ij} E_{\mu\nu}\label{form2} 
\end{eqnarray}
where the symmetric matrices $\beta_k\in \S^{n\times n}$ are defined by $\beta_k^{(j,i)}=\beta_k^{(i,j)}=\frac{1}{2}\gamma_k^{(i,j)}$ if $i\not= j$ and
$\beta_k^{(i,i)}=\gamma_k^{(i,i)}$.
\end{theorem}
\begin{proof}
The first formula follows from Corollary \ref{cor:Riesz} as
by \eqref{Rieszr}, the Riesz
representers are given by
\[
v_{\lambda_k^{(i,j)}}(x) = \sum_{1\le \mu\le \nu\le n}
\lambda_{k}^{(i,j)}(\Phi(\cdot,x)E^s_{\mu\nu})E^s_{\mu\nu}.
\]
By \eqref{action} we have
\[\left(\Phi(\cdot,x)E^s_{\mu\nu}\right)_{ij}
=
\sum_{k,\ell=1}^n \Phi(\cdot,x)_{ijk\ell} (E^s_{\mu\nu})_{k\ell}.
\]
For $\mu=\nu$ we have
\[\lambda_{k}^{(i,j)}(\Phi(\cdot,x)E^s_{\mu\mu})E^s_{\mu\mu}
=
F_k(\Phi(\cdot,x)_{\cdot,\cdot,\mu,\mu})_{ij}E_{\mu\mu}.
\]
For $\mu<\nu$ we have
\begin{eqnarray*}\lambda_{k}^{(i,j)}(\Phi(\cdot,x)E^s_{\mu\nu})E^s_{\mu\nu}
&=&\frac{1}{\sqrt{2}}\left(
  F_k(\Phi(\cdot,x)_{\cdot,\cdot,\mu,\nu})_{ij}+F_k(\Phi(\cdot,x)_{\cdot,\cdot,\nu,\mu})_{ij}
  \right)\frac{1}{\sqrt{2}}(E_{\mu\nu}+E_{\nu\mu}) 
\\
&=&\frac{1}{2}\left( F_k(\Phi(\cdot,x)_{\cdot,\cdot,\mu,\nu})_{ij}+F_k(\Phi(\cdot,x)_{\cdot,\cdot,\nu,\mu})_{ij}\right)(E_{\mu\nu}+E_{\nu\mu}).
\end{eqnarray*}
Hence, we have
\begin{eqnarray*}
v_{\lambda_k^{(i,j)}}(x) &=& \sum_{ \mu=1}^n
F_k(\Phi(\cdot,x)_{\cdot,\cdot,\mu,\mu})_{ij}E_{\mu\mu}\\ 
&&\mbox{} +
\frac{1}{2}
\sum_{\substack{ \mu,\nu=1\\\mu\not =\nu}}^n
[F_k(\Phi(\cdot,x)_{\cdot,\cdot,\mu,\nu})_{ij}+
F_k(\Phi(\cdot,x)_{\cdot,\cdot,\nu,\mu})_{ij}]E_{\mu\nu},
\end{eqnarray*}
which shows \eqref{form1}. To show \eqref{form2}, note that
by  \eqref{F:symm} we have
\begin{equation}\label{help}
F_k(\Phi(\cdot,x)_{\cdot,\cdot,\mu,\nu})_{ij}
=F_k(\Phi(\cdot,x)_{\cdot,\cdot,\nu,\mu})_{ji}. 
\end{equation}
To show \eqref{form2} it suffices to establish

\begin{multline*}
\sum_{i,j=1}^n\beta_k^{(i,j)} \sum_{\mu,\nu=1}^n
  F_k(\Phi(\cdot,x)_{\cdot,\cdot,\mu,\nu})_{ij} E_{\mu\nu}
=
 \sum_{ \mu=1}^n
\sum_{1\le i\le j\le n} \gamma_k^{(i,j)}
F_k(\Phi(\cdot,x)_{\cdot,\cdot,\mu,\mu})_{ij}E_{\mu\mu}\\
\mbox{} +
\sum_{\substack{ \mu,\nu=1\\\mu\not =\nu}}^n\sum_{1\le i\le j\le n} \gamma_k^{(i,j)}
\frac{1}{2}
\left[F_k(\Phi(\cdot,x)_{\cdot,\cdot,\mu,\nu})_{ij}+
F_k(\Phi(\cdot,x)_{\cdot,\cdot,\nu,\mu})_{ij}\right]E_{\mu\nu}
\end{multline*}
for $1\le k\le N$. We compare the expressions on both sides above for
each $E_{\mu\nu}$. For $\mu=\nu$ we have to show
\[
\sum_{i,j=1}^n\beta_k^{(i,j)} F_k(\Phi(\cdot,x)_{\cdot,\cdot,\mu,\mu})_{ij}
=
\sum_{1\le i\le j\le n} \gamma_k^{(i,j)} F_k(\Phi(\cdot,x)_{\cdot,\cdot,\mu,\mu})_{ij}.
\]
This is true, since for $i=j$ we have $\gamma_k^{(i,i)}=\beta_k^{(i,i)}$ and for
$i\not=j$ we have $F_k(\Phi(\cdot,x)_{\cdot,\cdot,\mu,\mu})_{ij} =
F_k(\Phi(\cdot,x)_{\cdot,\cdot,\mu,\mu})_{ji} $ by \eqref{help} and
$\frac{1}{2}\gamma_k^{(i,j)}=\beta_k^{(i,j)}=\beta_k^{(j,i)}$.

For $\mu\not=\nu$ we have to show
\begin{multline*}
\sum_{i,j=1}^n\beta_k^{(i,j)} F_k(\Phi(\cdot,x)_{\cdot,\cdot,\mu,\nu})_{ij}\\
=
\frac{1}{2}
\sum_{1\le i\le j\le n} \gamma_k^{(i,j)}
\left[
F_k(\Phi(\cdot,x)_{\cdot,\cdot,\mu,\nu})_{ij}+
F_k(\Phi(\cdot,x)_{\cdot,\cdot,\nu,\mu})_{ij}\right].
\end{multline*}
Again, this is shown using \eqref{help}  since for $i=j$ we have
$F_k(\Phi(\cdot,x)_{\cdot,\cdot,\mu,\nu})_{ii}
=F_k(\Phi(\cdot,x)_{\cdot,\cdot,\nu,\mu})_{ii} $ and
$\gamma_k^{(i,i)}=\beta_k^{(i,i)}$ and for 
$i\not=j$ we have $F_k(\Phi(\cdot,x)_{\cdot,\cdot,\mu,\nu})_{ji}
=F_k(\Phi(\cdot,x)_{\cdot,\cdot,\nu,\mu})_{ij} $ and
$\frac{1}{2}\gamma_k^{(i,j)}=\beta_k^{(i,j)}=\beta_k^{(j,i)}$. 
\end{proof}

The error estimate from Theorem \ref{th:error}, or more precisely
from Theorem \ref{th:errordynsys}, gives together with Theorem
\ref{est} our final result.

\begin{theorem} \label{thmfinal}
Let $f\in C^{\lceil\sigma\rceil}(\R^n;\R^n)$, $\sigma>n/2+1$. Let $x_0$ be an
exponentially stable equilibrium of $\dot x =f(x)$ with basin of
attraction $A(x_0)$. Let $C\in\S^{n\times n}$ be a positive definite
(constant) matrix and let $M\in C^{\sigma}(A(x_0),\S^{n\times n})$ be
the solution of (\ref{pde-const}) from Theorem \ref{est}. Let
$K\subseteq \Omega\subseteq A(x_0)$ be a positively invariant and
compact set, where $\Omega$ is open with Lipschitz boundary. Finally, let
$S$ be the optimal recovery from Theorem \ref{th:S}. Then, we have the error
estimate 
\[
\|F-S\|_{L_\infty(K;\S^{n\times n})} \le
  c\|F(M)-F(S)\|_{L_\infty(\Omega;\S^{n\times n})} \le C
  h_{X,\Omega}^{\sigma-1-n/2} 
\|M\|_{H^\sigma(\Omega;\S^{n\times n})}.
\]
for all $X\subseteq\Omega$ with sufficiently small $h_{X,\Omega}$. In
particular, $S$ itself is a contraction metric in $K$ provided $h_{X,\Omega}$ is
sufficiently small.
\end{theorem}
\begin{proof}
The error estimates and the linear independence of the $\lambda_k^{(i,j)}$
follow immediately from Theorem
\ref{th:errordynsys} with $V(x)=Df(x)\in
H^{\sigma-1}(\Omega;\R^{n\times n})$. To see that $S$ 
itself defines a contraction metric, we have to verify that $S$ is positive
definite and $F(S)$ is negative definite. We will do this only for $S$
as the proof for $F(S)$ is almost identical. The main idea here is
that the eigenvalues of symmetric matrix depend continuously on the
matrix values. To  be more precise, since  $M(x)$ is positive
definite for every 
$x\in K$ all its eigenvalues $\lambda_j(x)$, $1\le j\le n$ are
positive. If we order them by size, i.e. $0<\lambda_1(x)\le
\lambda_2(x)\le \ldots \lambda_n(x)$, then we have for $x,y\in K$, 
\[
|\lambda_j(x)-\lambda_j(y)|\le \|M(x)-M(y)\|
\]
for any natural matrix norm. Since $M$ is continuous, so is each
function $\lambda_j$. Since $K$ is compact, there is a
$\lambda_{\min}$ such that  $\lambda_j(x)\ge \lambda_{\min}>0$ for all
$1\le j\le n$ and all $x\in K$. If we now sort the eigenvalues
$\mu_j(x)$ of $S(x)$ in the same way, similar arguments as above show
\[
|\lambda_1(x)-\mu_1(x)|\le \|M(x)-S(x)\|\le C
h_{X,\Omega}^{\sigma-1-n/2}\|M\|_{H^\sigma(\Omega;\S^{n\times n})}
\]
Hence, if we choose $h_{X,\Omega}$ so small that the term on the right-hand
side becomes less then $\lambda_{\min}/2$, we see that
$\mu_1(x) \ge\lambda_{\min}/2$ for all $x\in K$, i.e. $S(x)$ is
also positive definite for all $x\in K$.
\end{proof}

While this result guarantees that $S(x)$ is eventually positive
definite for all $x\in K$, it does not provide us with an a priori
estimate on how small $h_{X,\Omega}$ actually has to be since we neither
know the constant $C>0$ nor the norm of the unknown function $M$.
Hence, in applications, we have to verify the positive definiteness
differently. 

\section{Example}
\label{sec4}

As an example we consider the linear system
\[
\dot{x}=-x+y, \qquad 
\dot{y}=x-2y,
\]
which was considered in \cite{Giesl-Hafstein-13-1} as a time-periodic example.
Note that the solution of the matrix equation (\ref{pde-const}) with
$C=-I$ is constant and can be easily calculated as
\[
M(x)=\left(\begin{array}{cc} 1&\frac{1}{2}\\[0.2cm]
  \frac{1}{2}&\frac{1}{2}\end{array}\right),
\]
which allows us to analyse the error to the exact solution.
Also note that any set of the form $K_c=[-c,c]^2$  with $c>0$ is
positively invariant. 
We have used grids of the form
$X_\alpha=\{(x,y)\in \R^2 :
x,y=-1,\ldots,-2\alpha,-\alpha,0,\alpha,2\alpha,\ldots,1\}$ with
$\alpha=1,\frac{1}{2},\frac{1}{2^2},\ldots,\frac{1}{2^5}$. 

As the RBF we have used  Wendland's $C^8(\R^2)$ function 
\[
\phi(r)=(1-cr)^{10} (2145 (cr)^4+2250 (c r)^3+1050 (cr)^2 +
250cr+25)_+
\]
 with $c=0.9$ which is a reproducing kernel in
$H^\sigma(\R^2)$ with $\sigma=5.5$, see \cite{Wendland-05-1}.

In each case we have calculated the errors
\begin{eqnarray*} 
e_\alpha&=&\max_{x\in X_{check}}\|S^\alpha(x)-M(x)\|_{\max}=\max_{x\in
  X_{check}}\max_{i,j=1,2}|S^\alpha_{ij}(x)-M_{ij}(x)|\\
e^s_\alpha&=&\max_{x\in  X_{check}}\|F(S^\alpha)(x)-F(M)(x)\|_{\max},\\
\end{eqnarray*}
with $X_{check}=\{(x,y)\in \R^2 :
x,y=-1+\frac{1}{2}\alpha_0,\ldots,-\frac{3}{2} \alpha_0,-\frac{1}{2}
\alpha_0,\frac{1}{2}\alpha_0,\frac{3}{2}\alpha_0,\ldots,1-\frac{1}{2}\alpha_0\}$
with $\alpha_0=\frac{1}{2^6}$. By Theorem \ref{thmfinal} we expect the
errors to 
behave like 
\[
\frac{e_{2\alpha}}{e_{\alpha}}\approx 2^{\sigma-1-n/2} = 2^{3.5}.
\]

Table \ref{table1} shows the above described errors for different
$\alpha$ as well as the expected ratios. 

\begin{table}[ht]
\begin{center}
\begin{tabular}{l||l|l||l|l}
$\alpha$ & $e_\alpha^s$ & $e_{2\alpha}^s/e_{\alpha}^s$ & $e_\alpha$
  &$e_{2\alpha}/e_{\alpha}$ \\\hline
1/2      & 2.5724     &         &  1.2334      &\\
1/4      & 1.2833     & 2.0045  &  0.9169      & 1.3452 \\
1/8      & 0.3516     & 3.6499  &  0.0124      & 73.9435 \\
1/16     & 0.0329     & 10.6838 &  5.6040e-4   & 22.1271 \\ 
1/32     & 0.0025     & 13.1918 &  1.6311e-5   & 34.3572         \\\hline
$2^{3.5}$   &            & 11.3137 &              & 11.3137 \\
\end{tabular}
\end{center}
\caption{Errors for various computation grids together with the error
  behaviour\label{table1}.}
\end{table}

Finally, we have fixed the grid to 
$X=\{(x,y)\in \R^2 : x,y=-4,-3.8,-3.6,\ldots,0,0.2,\ldots,4\}$ with
$N=1681$ points, and as each grid point requires $3$ variables of a
symmetric $2\times 2$ matrix, we solve a linear system with a
$5043\times 5043$ matrix. 
We need to check that the constructed matrix-valued function $S(x)$ is positive definite and $F(S)(x)$ is negative definite, where
$F(S)=Df(x)^T S(x)+S(x)Df(x)+S'(x)$. To check that a $2\times 2$
matrix $A$ is positive/negative definite it suffices to check that $\tr (A)$ is
positive/negative and $\det (A)$ is positive/$-\det (A)$ is
negative. The figures show $\tr F(S)(x)$, $-\det F(S)(x)$, which are
negative apart from small areas (Figure \ref{fig1_linear}), as well as
$\tr S(x)$, $\det S(x)$ which are positive (Figure
\ref{fig2_linear}). Figure \ref{fig3_linear}, left,  summarises the
results by displaying the grid points and the areas where the above
are zero. Figure \ref{fig3_linear}, right, illustrates the metric
$S(x)$ by plotting ellipses $\bx+\bv$ around $\bx$ with
$(\bv-\bx)^TS(\bx)(\bv-\bx)=$const. 

\begin{figure}[h!]
\includegraphics[height=6.75cm,width=6.75cm]{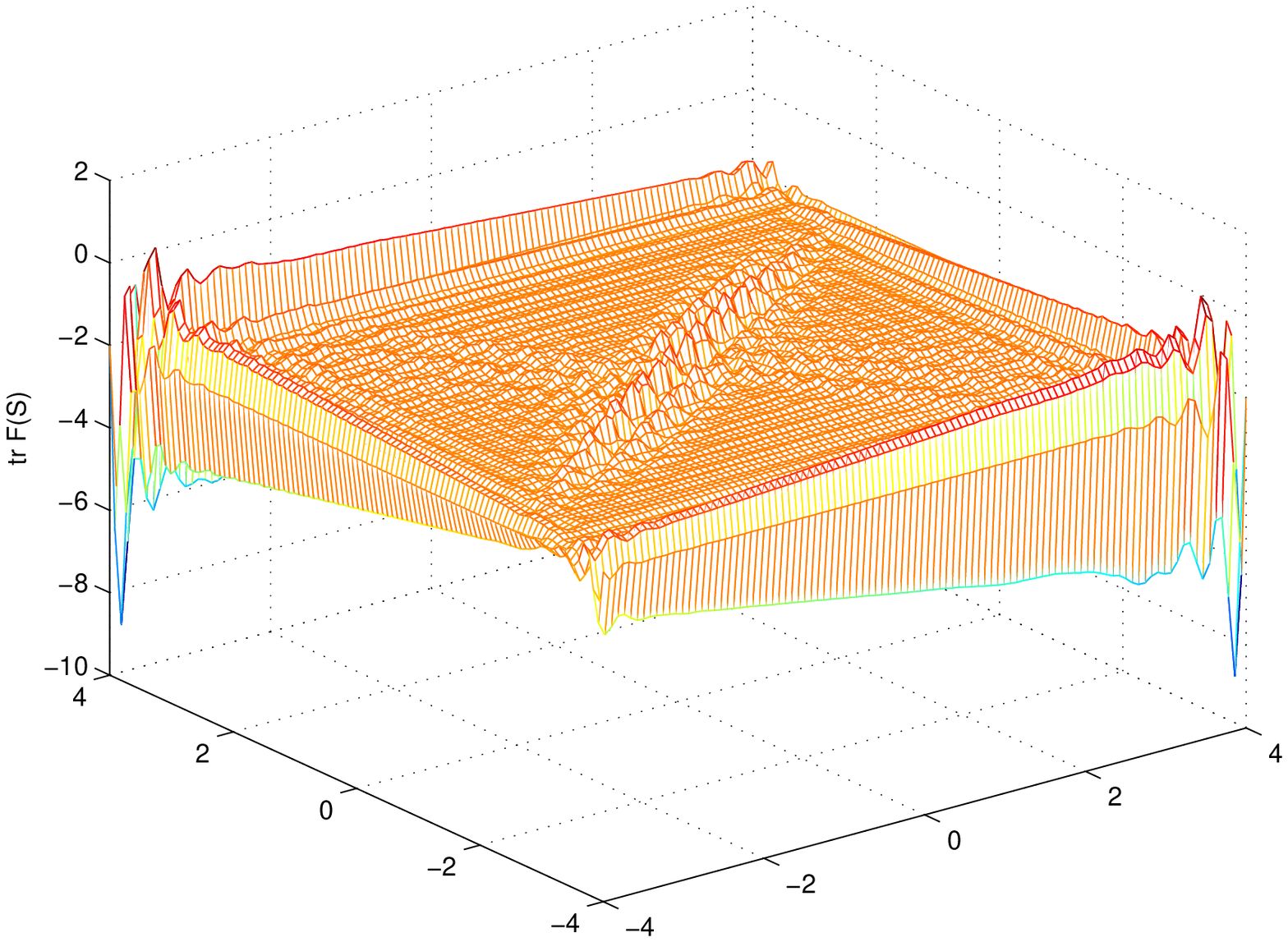}
\includegraphics[height=6.75cm,width=6.75cm]{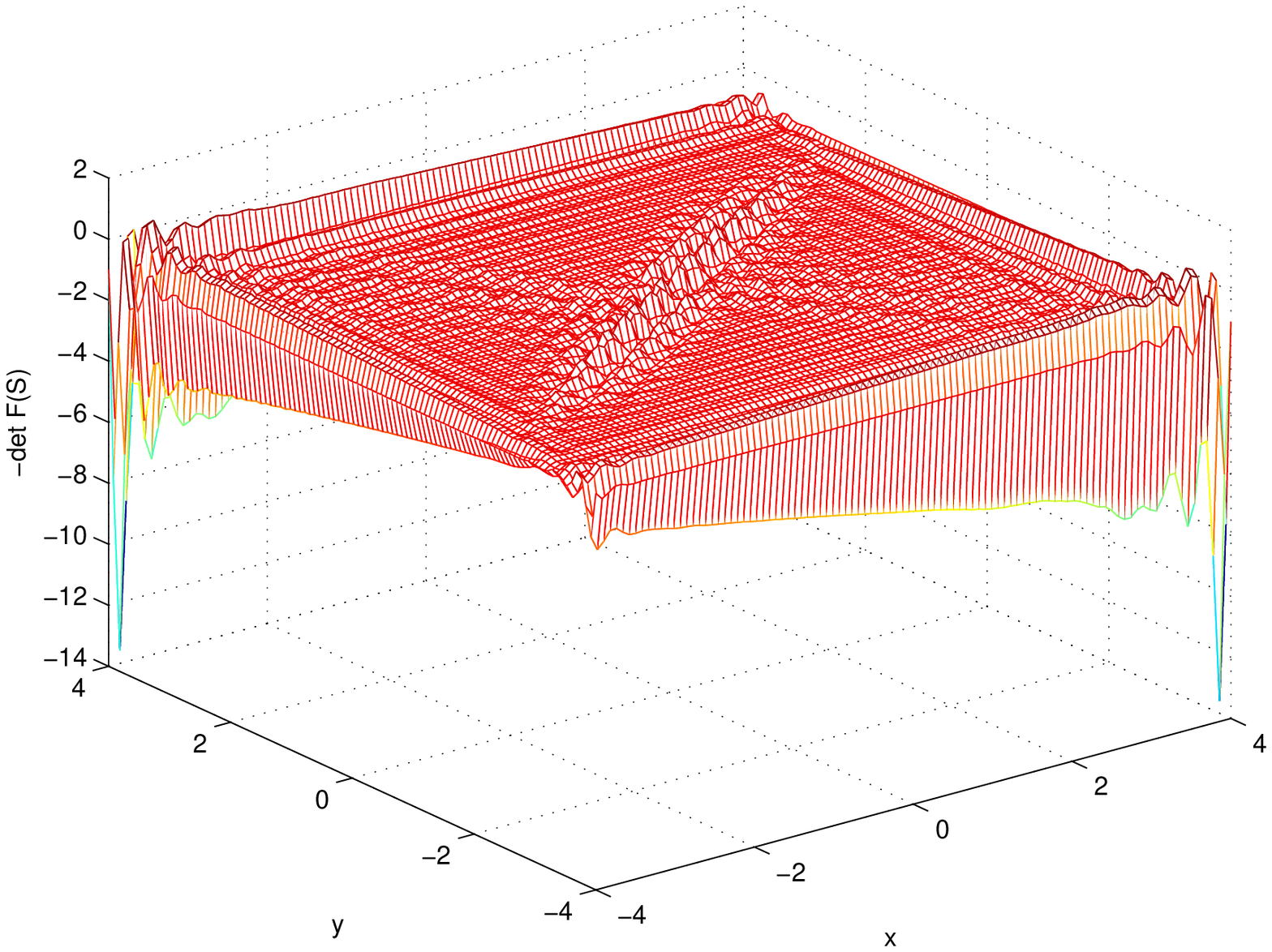}
\caption{Left: $\tr F(S)(x,y)$, right: $-\det F(S)(x,y)$. If both
  functions are negative, then $F(S)(x,y)$ is negative
  definite.}\label{fig1_linear} 
 \end{figure}

\begin{figure}[h!]
\includegraphics[height=6.75cm,width=6.75cm]{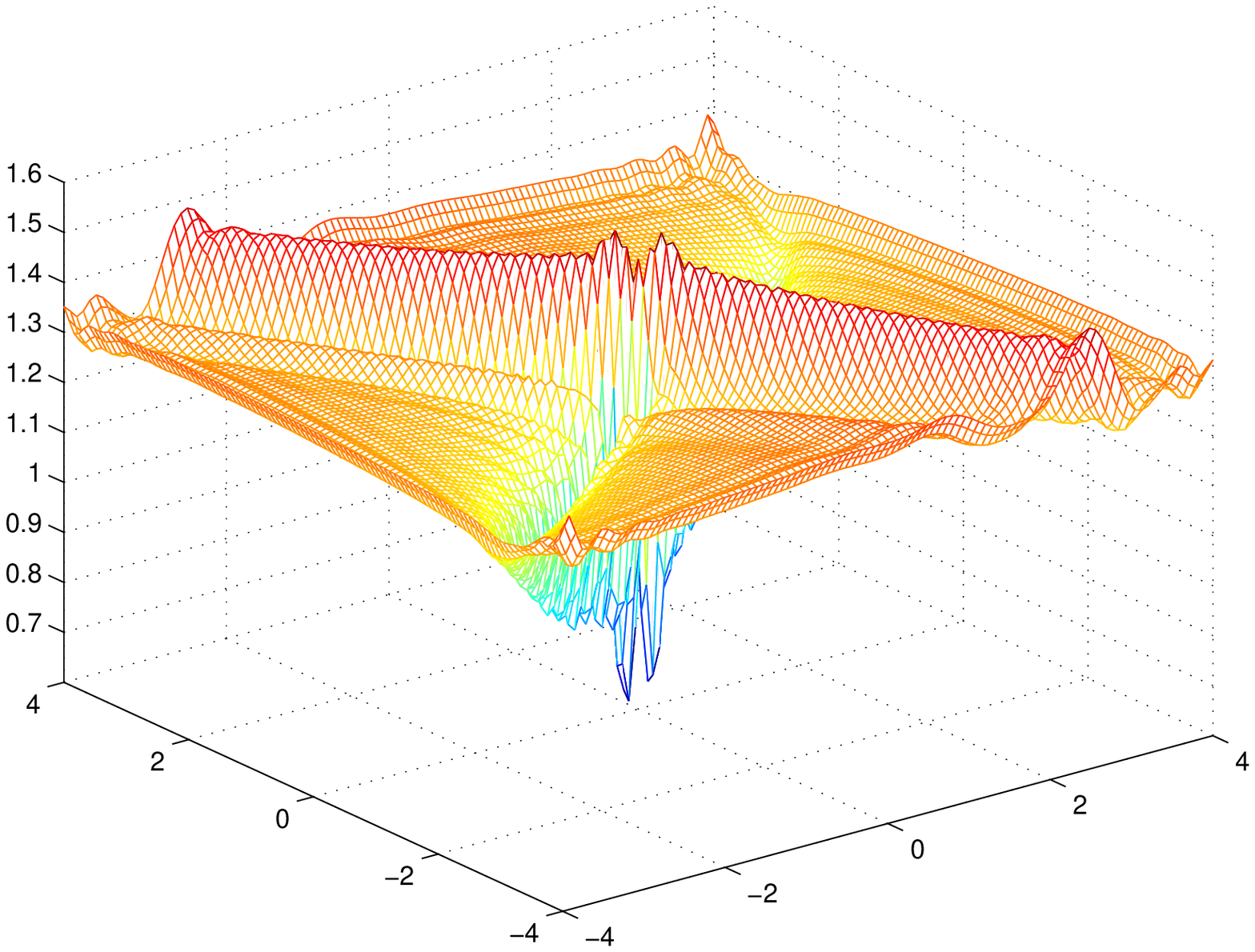}
\includegraphics[height=6.75cm,width=6.75cm]{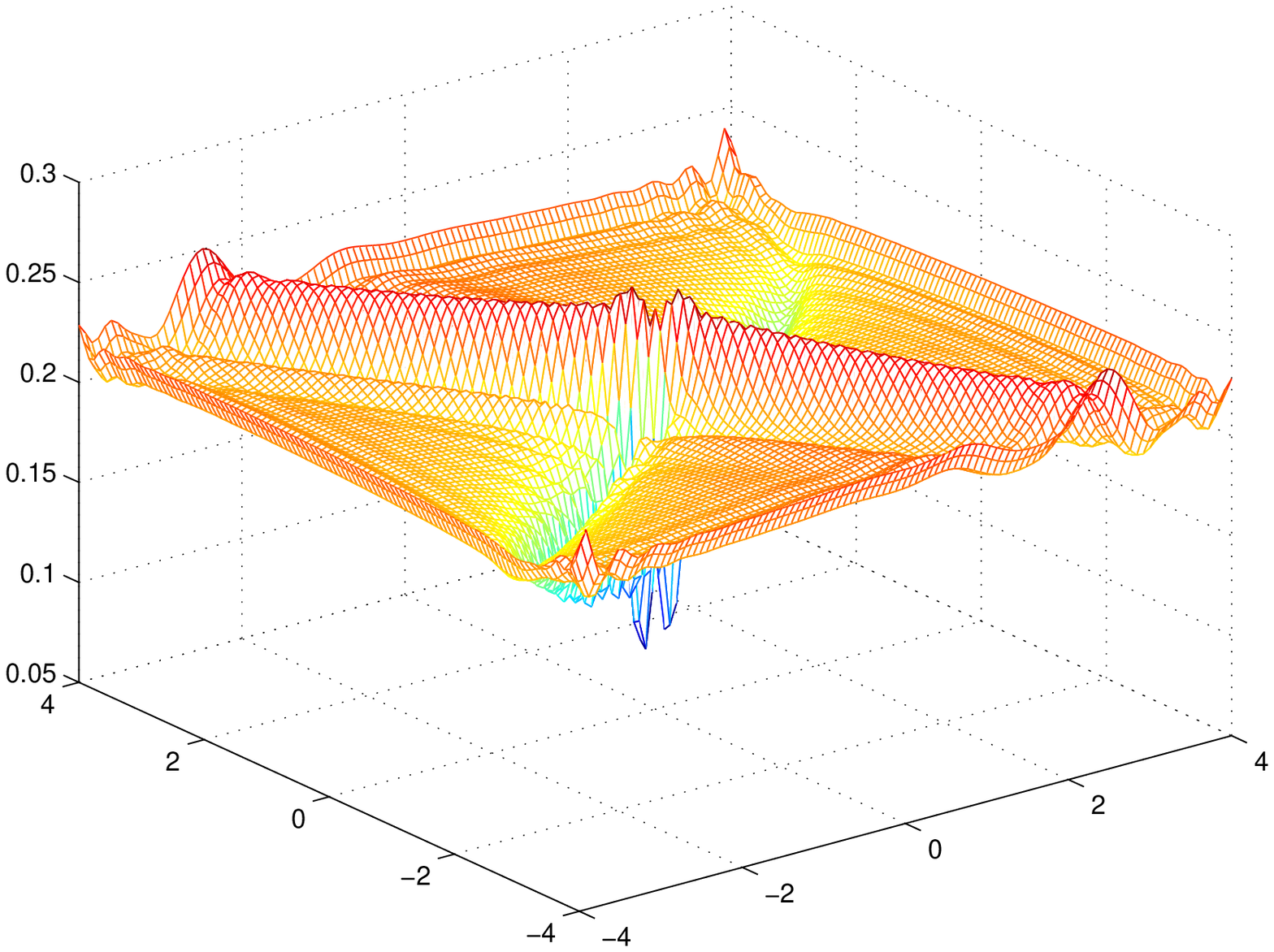}
\caption{Left: $\tr S(x,y)$, right: $\det S(x,y)$. If both functions
are positive, then $S(x,y)$ is positive definite.}\label{fig2_linear} 
\end{figure}

\begin{figure}[h!]
\includegraphics[height=6.75cm,width=6.75cm]{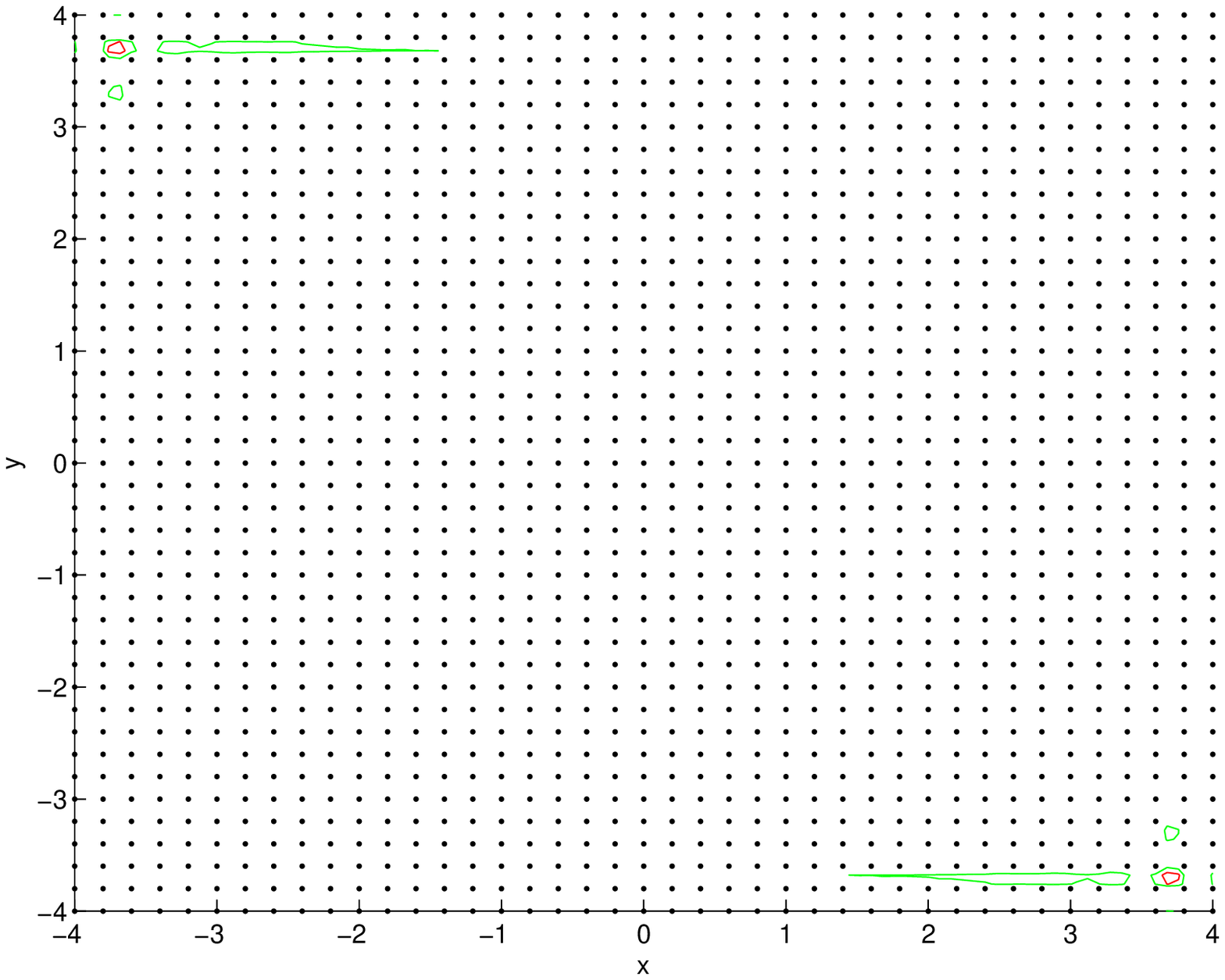}
\includegraphics[height=6.75cm,width=6.75cm]{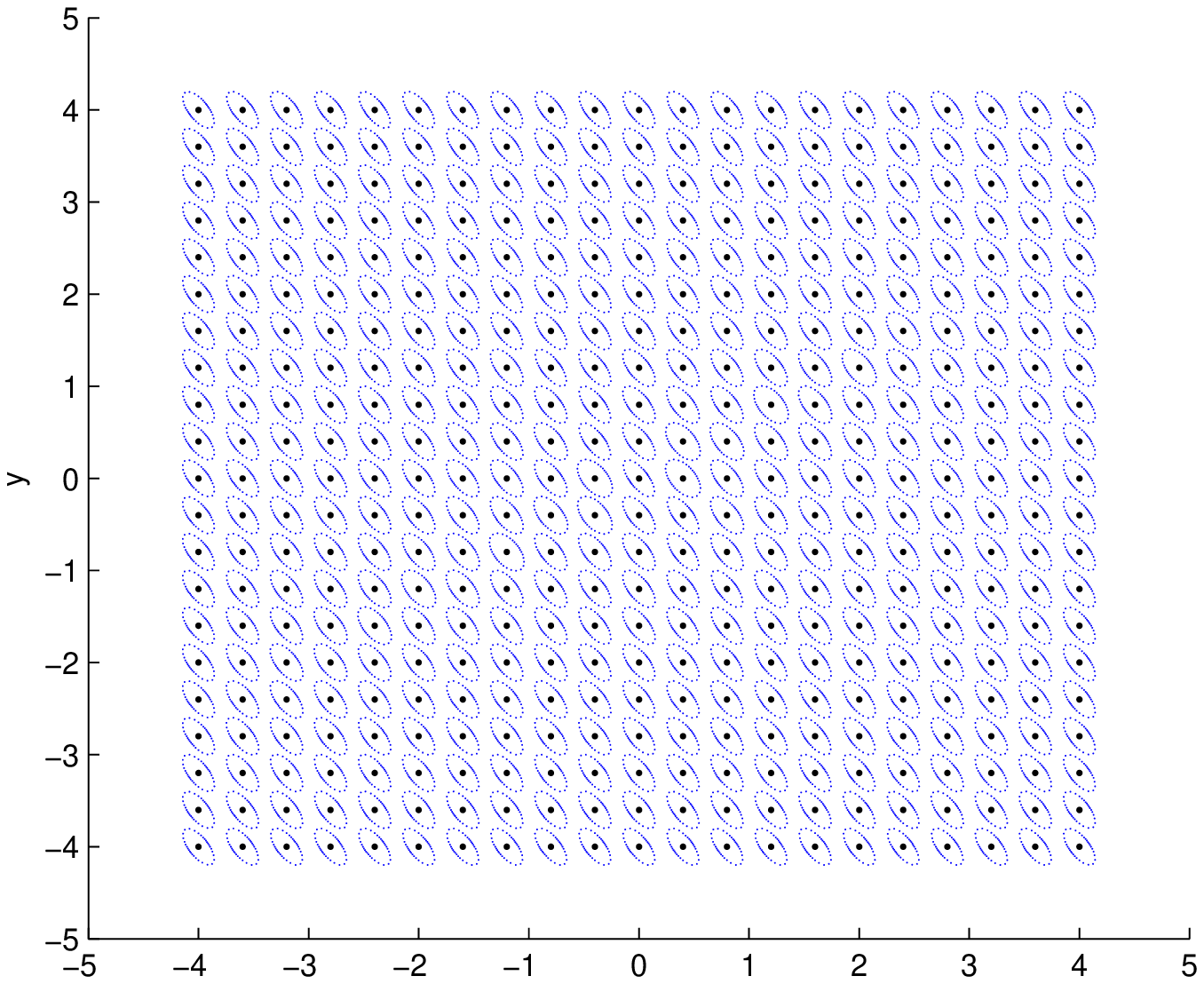}
\caption{Left: The points used for the RBF approximation together with
  the areas where $\tr F(S)(x,y)=0$ (red) and $\det F(S)(x,y)=0$
  (green). 
Right: To illustrate the approximation $S$, around some points $\bx$,
we have plotted the curve of equal distance with respect to metric
$S(\bx)$, in particular the set $\{\bx+\bv\mid
(\bv-\bx)^TS(\bx)(\bv-\bx)={\rm const}\}$.}\label{fig3_linear} 
\end{figure}

\bibliographystyle{siam}
\bibliography{../../../Utilities/rbf}

\begin{thebibliography}{10}

\bibitem{Amodei-97-1}
{\sc L.~Amodei}, {\em Reproducing kernels of vector-valued function spaces}, in
  Surface Fitting and Multiresolution Methods, A.~L. M\'ehaut'e, C.~Rabut, and
  L.~L. Schumaker, eds., Nashville, 1997, Vanderbilt University Press,
  pp.~17--26.

\bibitem{Aronszajn-50-1}
{\sc N.~Aronszajn}, {\em Theory of reproducing kernels}, Trans.\ Am.\ Math.\
  Soc., 68 (1950), pp.~337--404.

\bibitem{Aylward-etal-08-1}
{\sc E.~M. Aylward, P.~A. Parrilo, and J.-J. Slotine}, {\em Stability and
  robustness analysis of nonlinear systems via contraction metrics and {SOS}
  programming}, Automatica, 44 (2008), pp.~2163--2170.

\bibitem{Benbourhim-Bouhamidi-10-1}
{\sc M.~N. Benbourhim and A.~Bouhamidi}, {\em Meshless pseudo-polyharmonic
  divergence-free and curl-free vector fields approximation}, SIAM J.\ Math.\
  Anal., 42 (2010), pp.~1218--1245.

\bibitem{Berlinet-Thomas-Agnan-04-1}
{\sc A.~Berlinet and C.~Thomas-Agnan}, {\em Reproducing Kernel Hilbert Spaces
  in Probability and Statistics}, Springer, New York, 2004.

\bibitem{Borg-60-1}
{\sc G.~Borg}, {\em A condition for the existence of orbitally stable solutions
  of dynamical systems}, vol.~153 of Kungl. Tekn. H{\"o}gsk. Handl., Elander,
  1960.

\bibitem{Buhmann-03-1}
{\sc M.~D. Buhmann}, {\em Radial Basis Functions}, Cambridge Monographs on
  Applied and Computational Mathematics, Cambridge University Press, Cambridge,
  2003.

\bibitem{Cristianini-Shawe-Taylor-00-1}
{\sc N.~Cristianini and J.~Shawe-Taylor}, {\em An introduction to support
  vector machines and other kernel-based learning methods}, Cambridge
  University Press, Cambridge, 2000.

\bibitem{Cucker-Smale-02-1}
{\sc F.~Cucker and S.~Smale}, {\em On the mathematical foundation of learning},
  Bull.\ Amer.\ Math.\ Soc., 39 (2002), pp.~1--49.

\bibitem{Dick-etal-13-1}
{\sc J.~Dick, F.~Y. Kuo, and I.~H. Sloan}, {\em High-dimensional integration:
  The quasi-{M}onte {C}arlo way}, in Acta Numerica, A.~Iserles, ed., vol.~22,
  Cambridge University Press, 2013, pp.~133--288.

\bibitem{Fasshauer-07-1}
{\sc G.~Fasshauer}, {\em Meshfree Approximation Methods with {MATLAB}}, World
  Scientific Publishers, Singapore, 2007.

\bibitem{Flyer-Fornberg-11-1}
{\sc N.~Flyer and B.~Fornberg}, {\em Radial basis functions: Developments and
  applications to planetary scale flows}, Computers and Fluids, 46 (2011),
  pp.~23--32.

\bibitem{Fornberg-Flyer-15-1}
{\sc B.~Fornberg and N.~Flyer}, {\em Solving {PDEs} with radial basis
  functions}, in Acta Numerica, A.~Iserles, ed., vol.~24, Cambridge University
  Press, 2015, pp.~215--258.

\bibitem{Forni-Sepulchre-14-1}
{\sc F.~Forni and R.~Sepulchre}, {\em A differential {L}yapunov framework for
  contraction analysis}, IEEE Transactions on Automatic Control, 59 (2014),
  pp.~614--628.

\bibitem{Franke-Schaback-98-2}
{\sc C.~Franke and R.~Schaback}, {\em Convergence order estimates of meshless
  collocation methods using radial basis functions}, Adv.\ Comput.\ Math., 8
  (1998), pp.~381--399.

\bibitem{Fuselier-08-1}
{\sc E.~Fuselier}, {\em Sobolev-type approximation rates for divergence-free
  and curl-free {RBF} interpolants}, Math.\ Comput., 77 (2008), pp.~1407--1423.

\bibitem{Giesl-07-1}
{\sc P.~Giesl}, {\em Construction of global {L}yapunov functions using radial
  basis functions}, vol.~1904 of Lecture Notes in Mathematics, Springer-Verlag,
  Heidelberg, 2007.

\bibitem{Giesl-15-1}
\leavevmode\vrule height 2pt depth -1.6pt width 23pt, {\em Converse theorems on
  contraction metrics for an equilibrium}, J. Math. Anal. Appl., 424 (2015),
  pp.~1380--1403.

\bibitem{Giesl-Hafstein-13-1}
{\sc P.~Giesl and S.~Hafstein}, {\em Construction of a {CPA} contraction metric
  for periodic orbits using semidefinite optimization}, Nonlinear Anal., 86
  (2013), pp.~114--134.

\bibitem{Giesl-Hafstein-15-1}
{\sc P.~Giesl and S.~Hafstein}, {\em Computation and verification of {L}yapunov
  functions}, SIAM J. Appl. Dyn. Syst., 14 (2015), p.~1663–1698.

\bibitem{Giesl-Hafstein-15-2}
\leavevmode\vrule height 2pt depth -1.6pt width 23pt, {\em Review on
  computational methods for {L}yapunov functions}, Discrete and Continuous
  Dynamical Systems - Series B (DCDS-B), 20 (2015), pp.~2291 -- 2331.

\bibitem{Giesl-Wendland-07-1}
{\sc P.~Giesl and H.~Wendland}, {\em Meshless collocation: Error estimates with
  application to dynamical systems}, SIAM J.\ Numer.\ Anal., 45 (2007),
  pp.~1723--1741.

\bibitem{Giesl-Wendland-??-1}
\leavevmode\vrule height 2pt depth -1.6pt width 23pt, {\em Construction of a
  contraction metric by meshless collocation}.
\newblock Preprint Bayreuth/Sussex, 2016.

\bibitem{Hahn-67-1}
{\sc W.~Hahn}, {\em Stability of Motion}, Springer, Berlin, 1967.

\bibitem{Hartman-64-1}
{\sc P.~Hartman}, {\em Ordinary Differential Equations}, Wiley, New York, 1964.

\bibitem{Kansa-90-1}
{\sc E.~J. Kansa}, {\em Multiquadrics - {A} scattered data approximation scheme
  with applications to computational fluid-dynamics {I}. {S}urface
  approximations and partial derivative estimates}, Comput.\ Math.\ Appl., 19
  (1990), pp.~127--145.

\bibitem{Khalil-92-1}
{\sc H.~Khalil}, {\em Nonlinear systems}, Macmillan Publishing Company, New
  York, 1992.

\bibitem{Krasovskii-59-1}
{\sc N.~N. Krasovskii}, {\em Problems of the Theory of Stability of Motion},
  Mir, Moscow, 1959.

\bibitem{Leonov-etal-96-1}
{\sc G.~A. Leonov, I.~M. Burkin, and A.~I.Shepelyavyi}, {\em Frequency Methods
  in Oscillation Theory}, vol.~357 of Ser. Math. and its Appl., Kluwer, 1996.

\bibitem{Lowitzsch-02-1}
{\sc S.~Lowitzsch}, {\em Approximation and Interpolation Employing
  Divergence-Free Radial Basis Functions with Applications}, PhD thesis, Texas
  A\&M University, College Station, USA, 2002.

\bibitem{Micchelli-Pontil-05-1}
{\sc C.~A. Micchelli and M.~Pontil}, {\em Learning the kernel function via
  regularization}, J. of Mach. Learning Research, 6 (2005), pp.~1099--1125.

\bibitem{Narcowich-Ward-94-2}
{\sc F.~J. Narcowich and J.~D. Ward}, {\em Generalized {H}ermite interpolation
  via matrix-valued conditionally positive definite functions}, Math.\ Comput.,
  63 (1994), pp.~661--687.

\bibitem{Narcowich-etal-05-1}
{\sc F.~J. Narcowich, J.~D. Ward, and H.~Wendland}, {\em {S}obolev bounds on
  functions with scattered zeros, with applications to radial basis function
  surface fitting}, Math.\ Comput., 74 (2005), pp.~643--763.

\bibitem{Schaback-11-1}
{\sc R.~Schaback}, {\em The missing {W}endland functions}, Adv.\ Comput.\
  Math., 34 (2011), pp.~67--81.

\bibitem{Schaback-Wendland-06-1}
{\sc R.~Schaback and H.~Wendland}, {\em Kernel techniques: From machine
  learning to meshless methods}, in Acta Numerica, A.~Iserles, ed., vol.~15,
  Cambridge University Press, 2006, pp.~543--639.

\bibitem{Schoelkopf-Smola-02-1}
{\sc B.~Sch\"olkopf and A.~J. Smola}, {\em Learning with Kernels -- Support
  Vector Machines, Regularization, Optimization, and Beyond}, MIT Press,
  Cambridge, Massachusetts, 2002.

\bibitem{Steinwart-Christmann-08-1}
{\sc I.~Steinwart and A.~Christmann}, {\em Support Vector Machines}, Springer,
  2008.

\bibitem{Stenstrom-62-1}
{\sc B.~Stenstr{\"o}m}, {\em Dynamical systems with a certain local contraction
  property}, Math. Scand., 11 (1962), pp.~151--155.

\bibitem{Vapnik-98-1}
{\sc V.~Vapnik}, {\em Statistical Learning Theory}, John Wiley and Sons, 1998.

\bibitem{Wendland-95-1}
{\sc H.~Wendland}, {\em Piecewise polynomial, positive definite and compactly
  supported radial functions of minimal degree}, Adv.\ Comput.\ Math., 4
  (1995), pp.~389--396.

\bibitem{Wendland-99-2}
\leavevmode\vrule height 2pt depth -1.6pt width 23pt, {\em Meshless {G}alerkin
  methods using radial basis functions}, Math.\ Comput., 68 (1999),
  pp.~1521--1531.

\bibitem{Wendland-05-1}
\leavevmode\vrule height 2pt depth -1.6pt width 23pt, {\em Scattered Data
  Approximation}, Cambridge Monographs on Applied and Computational
  Mathematics, Cambridge University Press, Cambridge, UK, 2005.

\bibitem{Wendland-09-1}
\leavevmode\vrule height 2pt depth -1.6pt width 23pt, {\em Divergence-free
  kernel methods for approximating the {S}tokes problem}, SIAM J.\ Numer.\
  Anal., 47 (2009), pp.~3158--3179.

\end{thebibliography}

\end{document}